\documentclass[11pt,oneside,reqno]{amsart}
\usepackage[utf8]{inputenc}
\pdfoutput=1
\usepackage{parskip}
\usepackage{mathrsfs}
\usepackage[margin=1in]{geometry}
\usepackage{enumerate}
\usepackage{amsthm}
\usepackage{amsmath}
\usepackage{amsfonts}
\usepackage{amssymb}
\usepackage{graphicx}
\usepackage{color}
\usepackage{graphics}
\usepackage{eepic}
\usepackage{nicefrac}
\usepackage{bigints}
\usepackage{bbm}
\usepackage{array}
\usepackage{caption}

\makeatletter
\@addtoreset{equation}{chapter}
\makeatother
\usepackage{tikzsymbols}
\usepackage{booktabs}
\usepackage{hyperref}
\hypersetup{
 colorlinks   = true,
 urlcolor     = blue,
 linkcolor    = blue,
 citecolor   = red ,
 bookmarksopen=true
}

\newcommand{\ignore}[1]{}

\usepackage[colorinlistoftodos]{todonotes}

\newcommand{\be}{\begin{equation}}
\newcommand{\ee}{\end{equation}}

\usepackage{mathtools} 

\newtheorem{thm}{Theorem}[section]

\newtheorem{Lemma}[thm]{Lemma}

\newtheorem*{claim*}{Claim}

\theoremstyle{definition}

\theoremstyle{remark}

 \renewcommand\epsilon{\varepsilon}

\usepackage{palatino}
 \usepackage{color, xcolor, mathrsfs, graphicx, hyperref, framed}

\newcommand{\mb}{\mathbb}

\title{On the Number of components of random polynomial lemniscates}
\author{Subhajit Ghosh}

\address{Tata Institute of Fundamental Research, Centre for Applicable Mathematics, Bengaluru
560065, India}
\email{subhajitg@tifrbng.res.in}

\begin{document}

\begin{abstract}
A lemniscate of a complex polynomial $Q_n$ of degree $n$ is a sublevel set of its modulus, i.e., of the form $\{z \in \mb{C}: |Q_n(z)| <
t\}$ for some $t>0.$ In general, the number of connected components of this lemniscate can vary anywhere between 1 and $n$. In this paper, we study the expected number of connected components for two models of random lemniscates. First, we show that  lemniscates whose defining polynomial
has i.i.d. roots chosen uniformly from $\mb{D}$, has on average $\mathcal{O}(\sqrt{n})$ number of connected components. On the other hand if the i.i.d. roots are chosen uniformly from $\mb{S}^1$, we show
that the expected number of connected components, divided by n,
converges to $\frac{1}{2}$.
\end{abstract}

\maketitle

\section{Introduction}
Let $Q_n(z)$ be a monic polynomial of degree $n$ in the complex plane such that all its roots are contained within the closed unit disk $\overline{\mb{D}}$. That is, 
        \begin{align}\label{deterministic poly}
            Q_n(z):=\prod_{i=1}^n  (z-z_i),
        \end{align}
where $|z_j|\leq 1$, for $1\leq j\leq n$. We denote the unit lemniscate of $Q_n(z)$ by $\Lambda(Q_n):=\{z \in \mb{C} : |Q_n(z)| < 1\}.$ The quantity of interest is the number of connected components of $\Lambda$. The maximum principle implies that each connected component of the lemniscate must contain a zero of the polynomial; therefore, there are at most $n$ components.
In this paper, we investigate the number of components of a \emph{typical} lemniscate. Numerical simulations for random polynomials with roots chosen from the uniform probability measure on the unit disk $\mb{D}$, and on the circle $\mb{S}^1$ show a giant component alongside some tiny components (see Figures \ref{fig: comp in uni disk}, \ref{fig: comp in uni circ}). In this paper, we quantify this numerical observation.

\subsection{Motivation and Previous Results} The study of the metric and topological properties of polynomial lemniscates serves two main purposes. Firstly, it is the simplest curve with an algebraic boundary that is relevant to many problems in mathematical physics \cite{TopolofalgvarKKP,polylemtreesbraidsFM,lemofalgBF}. Secondly, polynomial lemniscates are used as a tool for approximating and analyzing complex geometric objects due to implications of Hilbert's lemniscate Theorem and its generalizations \cite{Ransford, SharpeningHLT}. For a more detailed exposition, please refer to \cite{KLM} and the corresponding references therein. Taking all these into account, in 1958, Erd\H{o}s, Herzog, and Piranian in \cite{metricEHP} studied the geometric and topological properties of polynomial lemniscates and posed a long list of open problems. One of the key motivations behind the work related to random polynomial lemniscates is to offer a probabilistic approach to the problems in \cite{metricEHP}. Krishnapur, Lundberg, and Ramachandran recently showed that the inradius of a random lemniscate whose defining polynomial
has roots chosen from a measure $\mu$ depends on the negative set of the logarithmic potential $U_\mu$. Lundberg, Epstein, and Hanin conducted a study on the lemniscate tree that encodes the nesting structure of the level sets of a random polynomial in \cite{treeEHL}. Lundberg and Ramachandran in \cite{arclengthLR} conducted a study on the \emph{Kac ensemble} and found that the expected number of connected components is asymptotically $n$. Lerario and Lundberg \cite{geometryLL} proved that for random rational lemniscates, which are defined as the quotient of two \emph{spherical random polynomials}, the average number of connected components is $\mathcal{O}(n)$. Later, Kabluchko and Wigman \cite{kabWig} discovered the exact asymptotics. Fyodorov, Lerario, and Lundberg studied the number of connected components of random algebraic
              hypersurfaces in \cite{compofgeohypersurfaceFLL}. In this article, we examine random polynomials with random roots, in contrast to random coefficients. Another stream of research on random polynomials includes studying the roots and critical points of random polynomials. In this work, we have made use of one such \emph{pairing} result due to Kabluchko and Seidel \cite{Kabluchko2019}, which states that for random polynomials whose roots are sampled from an appropriate probability measure $\nu$ supported within the unit disk, each root is associated with a critical point in close proximity. For more background, details and generalizations consult \cite{HaninPolynomial}, \cite{RourkNoah}, \cite{Criticalpointskabluchko}, \cite{SD}, \cite{PRzeroesofderivatives}, \cite{BLR}, \cite{michelen2022zeros}, \cite{almostsureAMP}, \cite{hoskins2021dynamics}, \cite{commonlimithu2017}, \cite{charpolyofmatrixRourke}. To find related research on meromorphic functions and Gaussian polynomials, please refer to \cite{Haninmeromorphic}, \cite{HaninGassianpoly}. We emphasize the fact that such pairing phenomena are exclusive to random polynomials. The analogous result in the deterministic setting is Sendov's conjecture \cite{Sendovconj}, which was recently proven by Tao in \cite{Tao} for all polynomials of sufficiently large degree.

\subsection{Main Results}In all the theorems we have the following setting.

\textbf{Setting and notations:} Let $\{X_i\}_{i=1}^\infty$ be a sequence of i.i.d. random variables with law $\mu$, supported in the closed unit disk. Consider the sequence of random polynomials defined by 
\begin{align}\label{01}
    P_n(z) :=\prod_{i=1}^n  (z-X_i)
\end{align}
and its lemniscate $$\Lambda_n:=\Lambda(P_n)=\{z \in \mb{C} : |P_n(z)| <1\}.$$ We denote by  $C(\Lambda_n)$ the number of connected components of the lemniscate $\Lambda_n$. Throughout the paper, we denote by C a positive numerical constant whose values may vary from line to line. For a set $S \subset \mb{C}$, we denote by $|S|$ the cardinality of the set $S$. The following are the main theorems of this paper. 


    \begin{figure}
    \centering
    \includegraphics[scale=.1725]{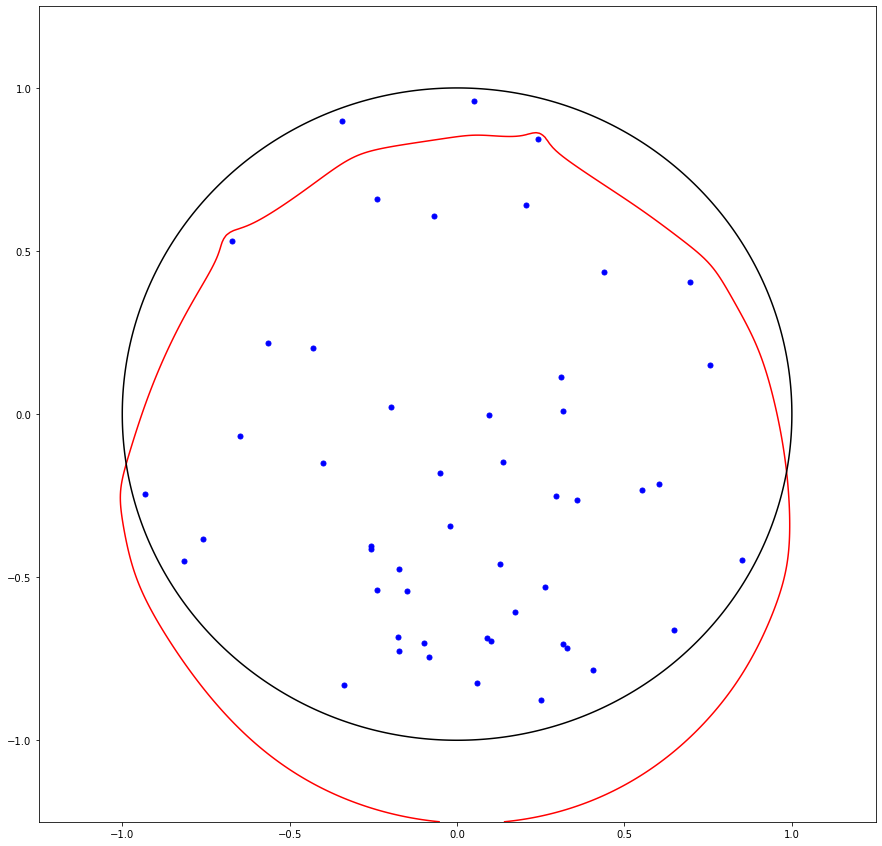}
    \includegraphics[scale=.1725]{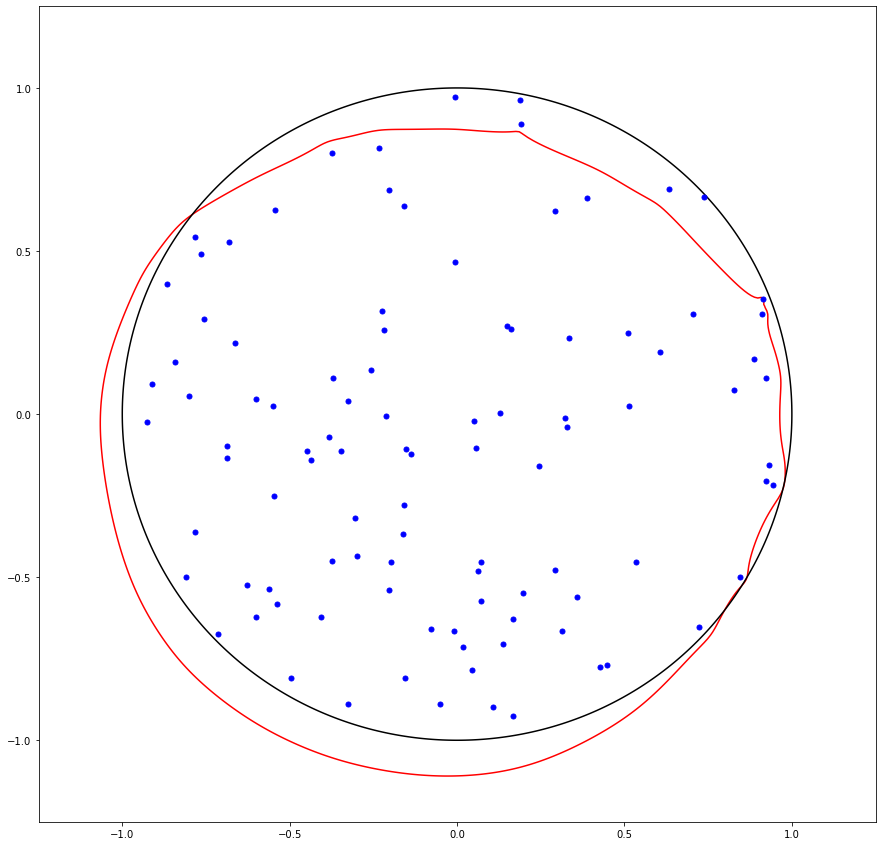}
    \includegraphics[scale=.1725]{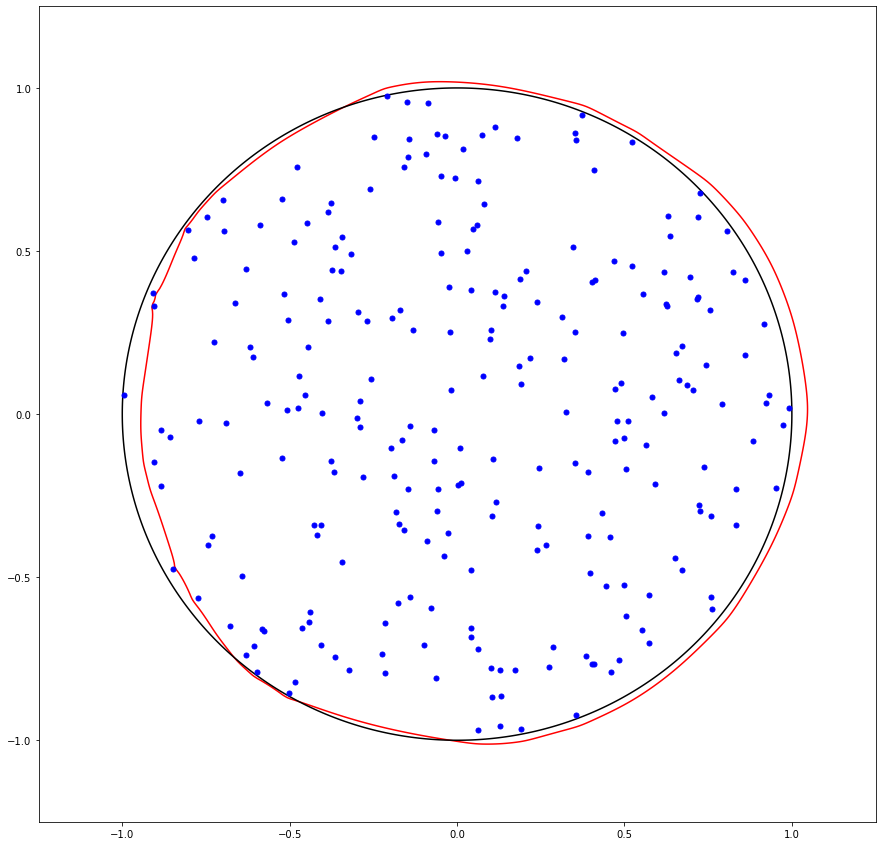}
    \caption{Lemniscates of degree n = 50, 100, 250 with zeros sampled uniformly from the open unit disk.}
    \label{fig: comp in uni disk}
    \end{figure}
    \begin{figure}
    \centering
    \includegraphics[scale=.1725]{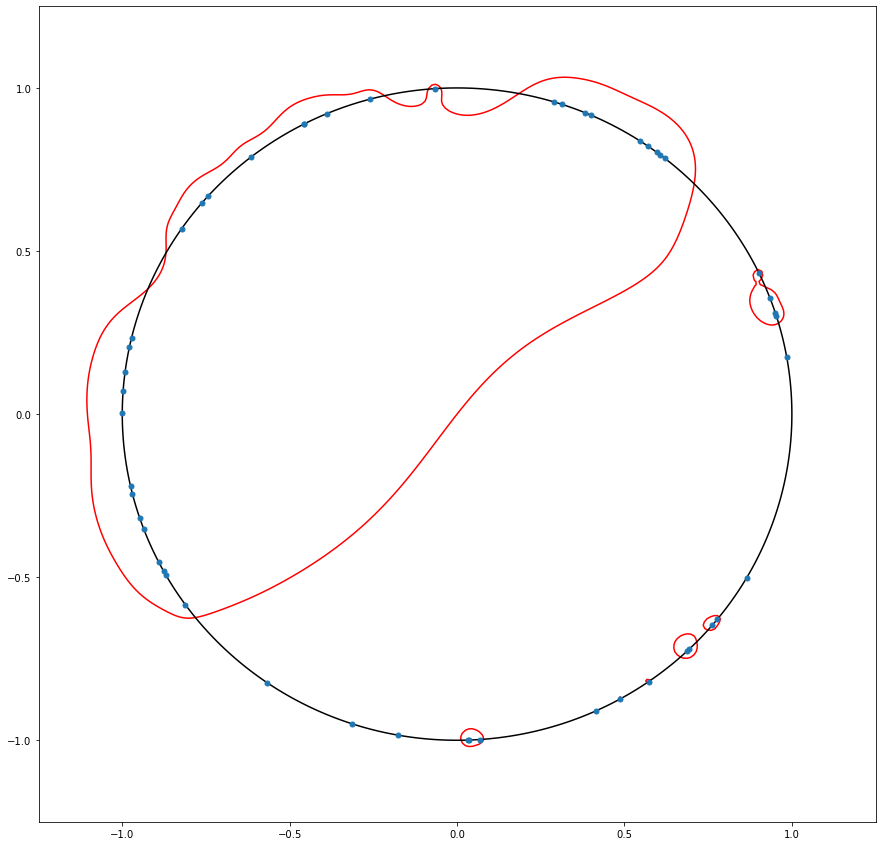}
    \includegraphics[scale=.1725]{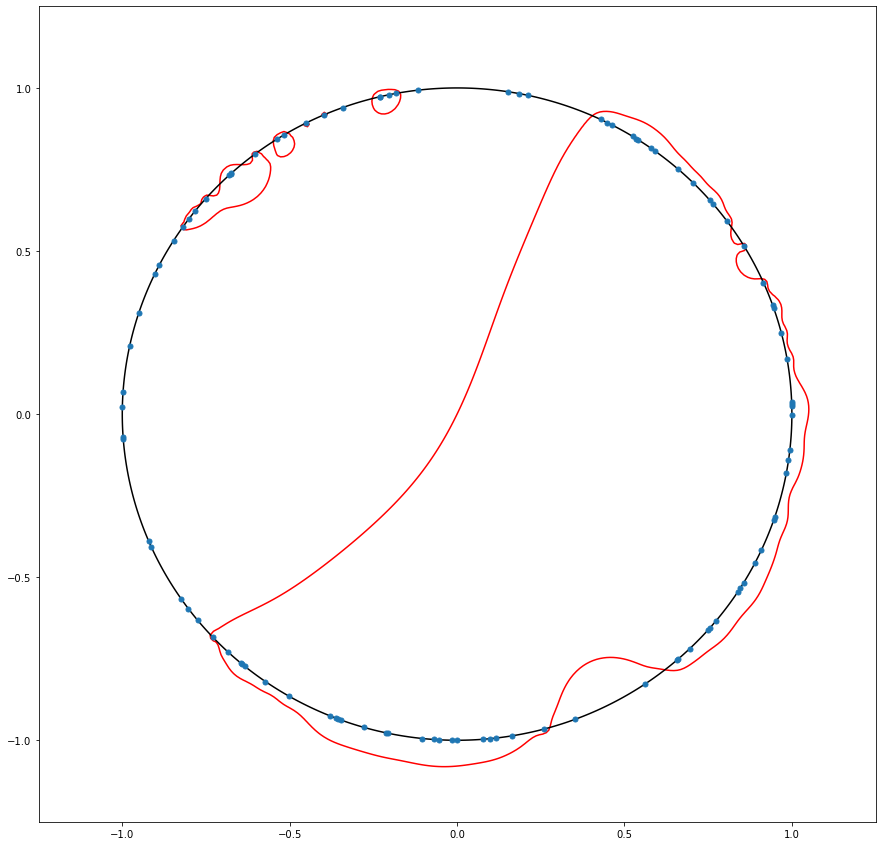}
    \includegraphics[scale=.1725]{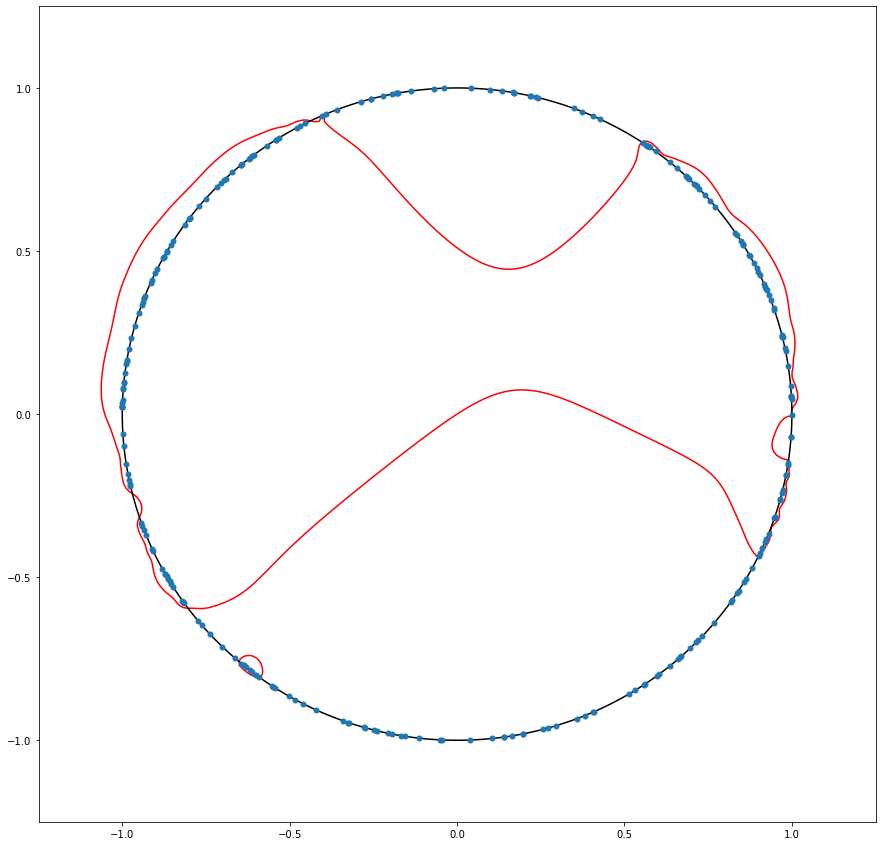}
    \caption{Lemniscates of degree n = 50, 100, 250 with zeros sampled uniformly from the unit circle. A unit circle is also plotted for reference in each case.}
    \label{fig: comp in uni circ}
    \end{figure}

        \begin{thm}\label{component uniform on disk}
             Let $\mu$ be the probability measure distributed uniformly in the unit disk $\mb{D}$. Then there exist absolute constants $C_1, C_2 > 0$ such that for all large $n$ we have 
            \begin{align*}
              C_1\sqrt{n}\leq {\mathbb{E}[C(\Lambda_n)]} \leq C_2\sqrt{n}.
            \end{align*}
        \end{thm}
        \begin{thm}\label{uniform on circle}
              Let $\mu$ be the probability measure distributed  uniformly in the unit circle $\mb{S}^1$. Then
            \begin{align*}
              \lim_{n \to \infty}\frac{\mathbb{E}[C(\Lambda_n)]}{n} =  \frac{1}{2} .
            \end{align*}
        \end{thm}
\subsection{Remarks}
What happens if we choose $\mu$ to be the uniform measure on $r\mathbb{D}$ or $r\mb{S}^1$? Let us consider the uniform probability measure on $r\mb{S}^1$ say $\mu_r$. Then it is easy to show that the logarithmic potential is
\begin{align}
    U_{\mu_r}(z)=
        \begin{cases}\label{potential}
            \log|z|  & \mbox{ if } |z| \geq r,\\
            \log r   & \mbox{ if } |z| < r.
        \end{cases}
\end{align}
\begin{description}
    \item[Case 1 $(r<1)$] In this case, the potential \eqref{potential} is negative in the whole unit disk. Therefore the set $r\mathbb{D}$ is enclosed within the lemniscate by Theorem 1.1 in \cite{KLM}, resulting in a single connected component with overwhelming probability.
    \item[Case 2 $(r>1)$] In this case, the potential \eqref{potential} is positive in the entire complex plane therefore we get with overwhelming probability, $n$ components for the lemniscate, by the implications of Theorem 1.3 of \cite{KLM}.
\end{description}
 So in some sense, $r=1$ is the critical case in this model. A similar analysis for the uniform probability measure on $r\mathbb{D}$ is done in \cite{KLM}, example-1.7. See Figure \ref{fig: comp in r circ}, \ref{fig: comp in r disk}. The above results and the results in this paper are summarized schematically in Table \ref{tab:components}.
 \begin{figure}
    
    \centering
    \includegraphics[scale=.25]{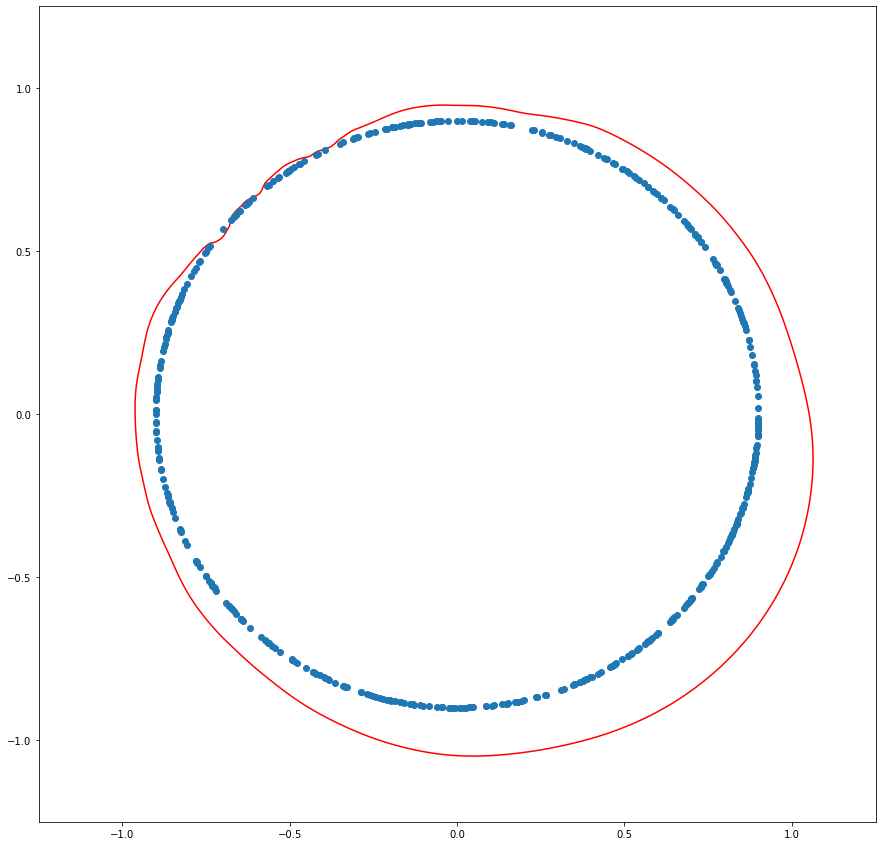}
    \includegraphics[scale=.25]{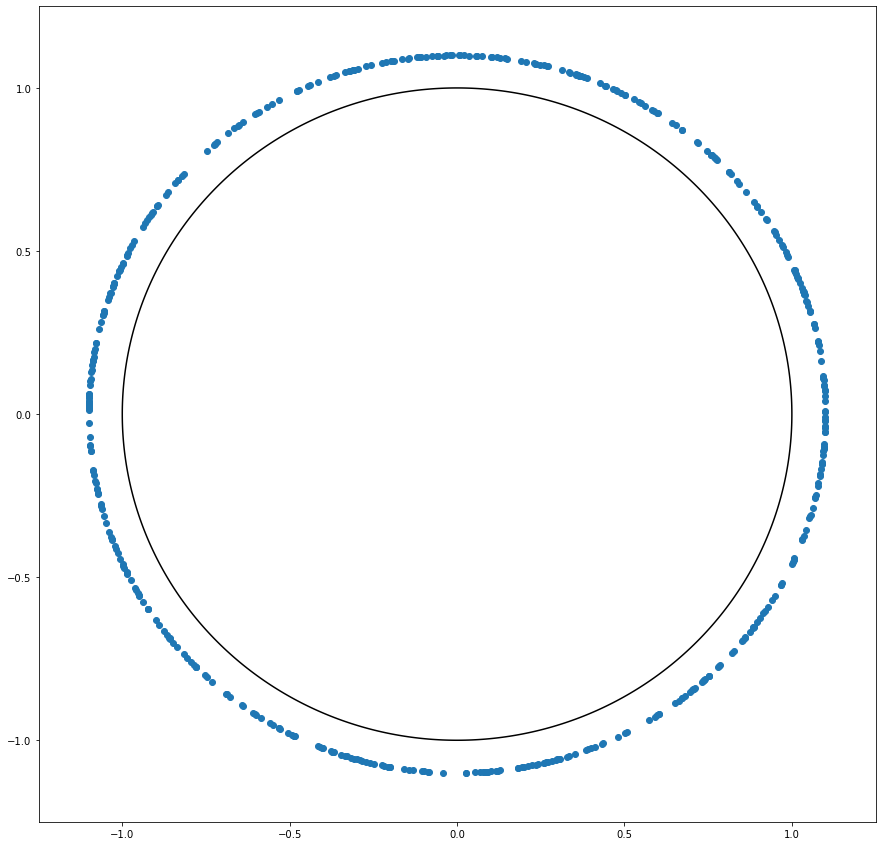}
    \caption{Lemniscates of degree n = 500 with zeros sampled uniformly from $r\mb{S}^1$, for $r=0.9$ and $1.1$ respectively.}\label{fig: comp in r circ}
    \end{figure}
    \begin{figure}
    \centering
    \includegraphics[scale=.1725]{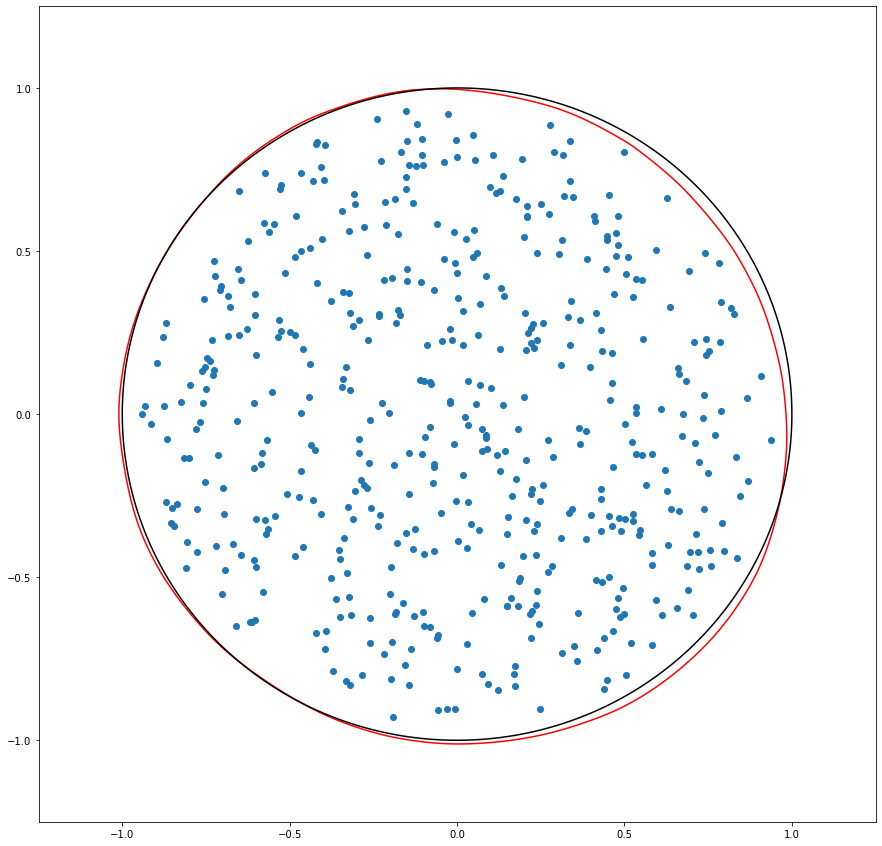}
    \includegraphics[scale=.1725]{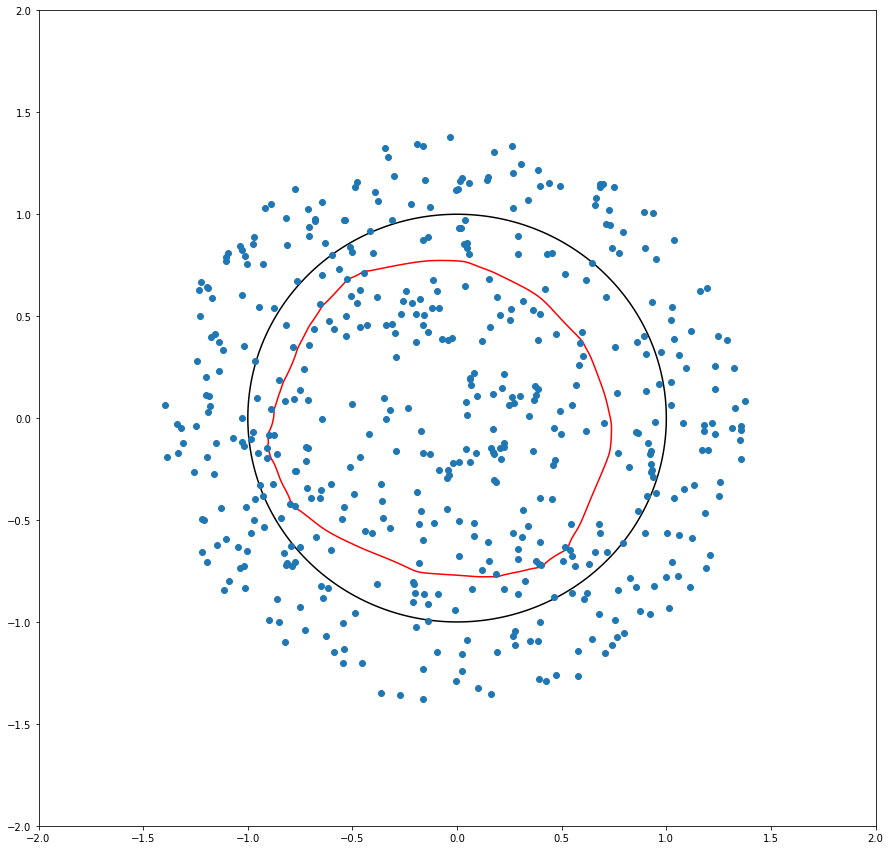}
    \includegraphics[scale=.1725]{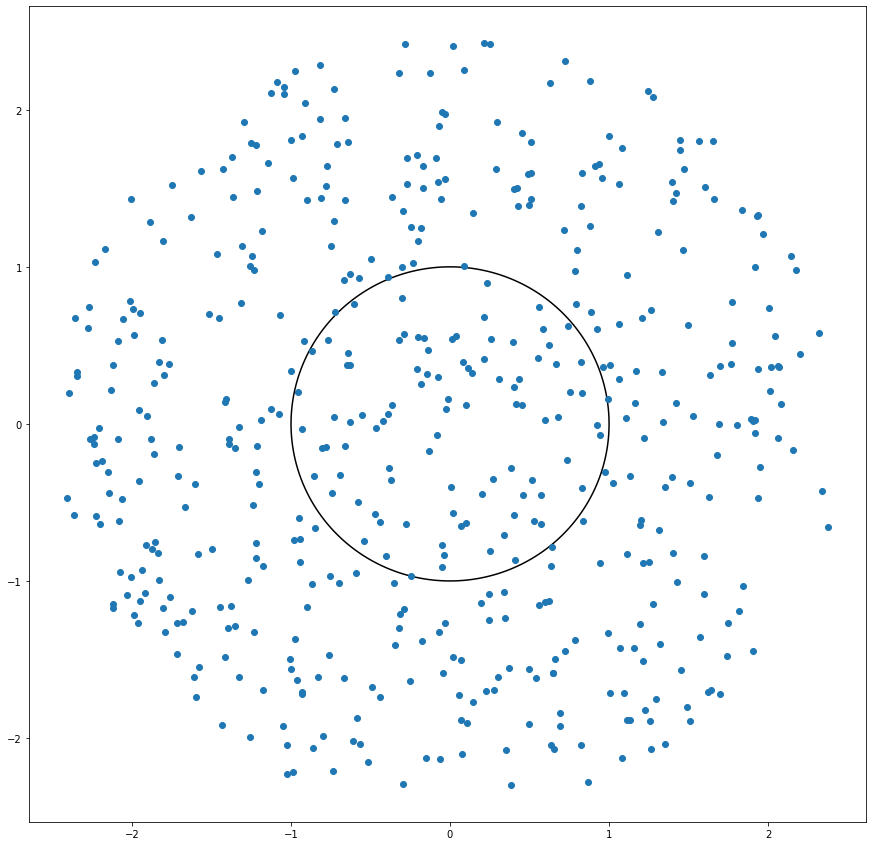}
    \caption{Lemniscates of degree n = 500 with zeros sampled uniformly from $r\mb{D}$ for $r=0.95, 0.85\sqrt{e},$ and $1.5\sqrt{e}$ respectively.}\label{fig: comp in r disk}
\end{figure}

\begin{table}[ht]
\centering
\begin{tabular}{||c|c|c|c|c|c||}
\hline
\textbf{} & $\boldsymbol{\mu}$ & $\mathbf{r < 1}$ & $\mathbf{r = 1}$ & $\mathbf{\sqrt{e} \geq r > 1}$ & $\mathbf{r > 1}$ \\
\hline
$\boldsymbol{\mathbb{E}[C(\Lambda_n)]}$ & Uniform probability measure in $r\mathbb{D}$ & $1$ & $\mathcal{O}(\sqrt{n})$ & $C_rn$ & $n$ \\
\hline
$\boldsymbol{\mathbb{E}[C(\Lambda_n)]}$ & Uniform probability measure in $r\mathbb{S}^1$ & $1$ & $\frac{n}{2}$ & $n$ & $n$ \\
\hline
\end{tabular}
\caption{Asymptotics of Expected No. of Components for Different Values of $r$.}
\label{tab:components}
\end{table}

\subsection{Heuristics and ideas of proof} We will now provide an overview of the underlying heuristics behind our results. In the first model, which involves random polynomials with uniformly chosen roots from $\mb{D}$, the potential $U_{\mu}(z)$ is negative throughout the unit disk. By writing $\log|P_n(z)|= \sum_{i=1}^n \log|z-X_i|$ as the sum of independent random variables with mean $U_{\mu}(z)$, we employ various concentration estimates to analyze the behavior of $|P_n(z)|$. Since the sum of i.i.d. random variables concentrate near its mean which is negative, most of the region within the disk, away from the boundary, lies inside the lemniscate. It is only near the boundary, where the potential approaches zero, isolated components are formed due to the fluctuations governed by the \emph{Central Limit Theorem}, resulting in $\mathcal{O}(\sqrt{n})$ many components. In the other model, i.e., random polynomial with roots chosen uniformly on the circle, the potential is zero in the whole disk. The probability of any point on $\mb{S}^1$ being inside the lemniscate is close to $\frac{1}{2}$. Therefore, if we start with $P_n$ and introduce a new root $X_{n+1}$ to build $P_{n+1},$ $X_{n+1}$ will land outside $\Lambda_n$ with probability approximately $\frac{1}{2}$, forming an isolated component. Therefore, on average, we get approximately $\frac{n}{2}$ components. In both models, we establish the lower bound by estimating the number of isolated components. To determine the upper bound for the disk case, we utilize an analytical characterization for the number of components (see Lemma \ref{Components and critical points}), which asserts that the number of components is one more than the number of critical points whose critical value is larger or equal to $1$. To determine the number of such critical points, we employ a \emph{pairing} result from \cite{Kabluchko2019} to associate critical points with roots with some desired properties. The number of such roots yields the desired upper bound. However, in the other case, the pairing phenomena does not occur. There we establish the upper bound by showing that the number of components possessing fewer than $n^{\epsilon}$ roots, when divided by $n$, tends towards $\frac{1}{2}$, for sufficiently small $\epsilon$.

\section{Preliminary Lemmas}
Before delving into the proofs of the main theorems, we gather preliminary theorems and lemmas that are utilized repeatedly in both theorems.
    \begin{thm}\textbf{(Berry-Esseen)}\label{BE}
        Let $X_1,X_2,...$ be i.i.d. random variables with $\mb{E}X_i=0,\mb{E}X_i^2=\sigma^2,$ and $\mb{E}|X_i|^3=\rho <\infty .$ If $F_n(x)$ is the distribution function of $\frac{(X_1+...+X_n)}{\sigma \sqrt{n}}$ and $\Phi(x)$ is the standard normal distribution, then
        \begin{align}
            |F_n(x)-\Phi(x)| \leq \frac{3\rho}{\sigma^3 \sqrt{n}}.
        \end{align}
    \end{thm}
    \noindent The proof of Theorem \ref{BE} can be found in \cite{durrett2019probability} Theorem 3.4.17.
    \begin{thm}\label{Ben} \textbf{[Bennett's inequality]} Let $Y_1,Y_2,...,Y_n$ be independent random variables with finite variance such that  $\forall$ $ i \leq n$, $Y_i \leq b$, for some $b > 0$ almost surely. Let
            \begin{align*}
                S=\sum_{i=1}^n\left(Y_i-\mathbb{E}[Y_i]\right),
            \end{align*}
        and $\nu = \sum_{i=1}^n\mathbb{E}[{Y_i}^2] .$ Then for any $t >0,$ we have
            \begin{align*}
                \mathbb{P}(S > t)\leq \exp\left(-\frac{\nu}{b^2}h\Big(\frac{bt}{\nu}\Big)\right),
            \end{align*}
        where $h(u) = (1 + u) \log(1 + u)- u \textit{, for } u > 0.$
    \end{thm}
    For the proof of this concentration inequality and other similar results see \cite{BLMbook}.
    \begin{Lemma}\label{log moment bound for uniform on cicle}
         Let $X$ be a random variable taking values in $\overline{\mb{D}}$ with law $\mu$. Assume that for all $z \in \mb{D}, r\leq 2$, there exist constants $ \epsilon,M_1,M_2 \in (0,\infty)$ such that $\mu$ satisfies 
         \begin{align}\label{lemma hypo}
             M_1r^{\epsilon}  \leq \mu(B(z,r))\leq M_2r^{\epsilon}.
         \end{align}
         Fix p, and define the function $F_p(z):= \mb{E}\Big[\big|\log|z-X|\big|^p\Big]:\mb{D} \to \mb{R}$. Then, there exist constants $C_1,C_2$ depending on $p,\epsilon,M_1,M_2$, such that
            \begin{align}\label{log moment ineq}
                C_1 \leq \inf_{z\in \mathbb{D}}F_p(z) \leq  \sup_{z\in \mathbb{D}}F_p(z) \leq  C_2.
            \end{align}
    \end{Lemma}
    \begin{proof}
            We will utilize the layer cake representation and write
                \begin{align*}
                        \mb{E}\left[\big|\log{|z-X|}\big|^p\right]&=p\int_0^\infty t^{p-1}\mb{P}\left(\big|\log{|z-X|}\big|>t\right) dt \\
                        &=p\int_0^{2}t^{p-1}\mb{P}\left(\big|\log{|z-X|}\big|>t\right)dt +p\int_{2}^\infty t^{p-1}\mb{P}\left(\big|\log{|z-X|}\big|>t\right)dt.
                \end{align*}
            In the second integral, notice that $(\log{|z-X|})^{+}<2$, therefore, probability is non zero when $\log{|z-X|}$ is negative. Taking this into account and using the upper bound in \eqref{lemma hypo}
                \begin{align*}
                   \mb{E}\left[\big|\log{|z-X|}\big|^p\right] &\leq p \int_0^{2} t^{p-1} d t + pM_2\int_{2}^\infty t^{p-1}e^{-t\epsilon}dt\\[.35em]
                    &\leq p2^{p+1}\left(1+C\left(\epsilon\right)M_2\right).
                \end{align*}
                   
            The lower bound follows similarly using the left inequality in \eqref{lemma hypo} along with the layer cake representation.
    \end{proof}
    \begin{Lemma}\label{moment bound for uniform on disk}
     Let $X$ be a uniform random variable on the open unit disk $\mathbb{D}$. For $ p<2 $, there exists a constant $C_p$ such that
    \begin{align}
            \mb{E}\left[\frac{1}{|z-X|^p}\right] \leq C_p.
        \end{align}
    \end{Lemma}
    \begin{proof}
    This proof is again based on the layer cake representation.
        \begin{align*}
            \mb{E}\left[\frac{1}{|z-X|^p}\right]&=\bigintsss_0^\infty \mb{P}\left(\frac{1}{|z-X|^p}>t\right) dt \\
            &=\bigintsss_0^\infty \mb{P}\left(|z-X|<\frac{1}{t^{1/p}}\right) dt\\
            &=\bigintsss_0^{2}\mb{P}\left({|z-X|}<\frac{1}{t^{1/p}}\right)dt +\bigintsss_{2}^\infty \mb{P}\left({|z-X|}<\frac{1}{t^{1/p}}\right)dt \\
            &=\int_0^{2}  dt +\int_{2}^\infty t^{-2/p}dt\\
            &\leq \left( 2+\frac{p}{2-p}2^{\frac{p-2}{p}}\right).\qedhere
        \end{align*}
    \end{proof}
    \begin{Lemma}(\textbf{Distance between the roots})\label{distance between roots}
        Let $\{X_i\}_{i=1}^\infty$ be a sequence of i.i.d. random variables with law $\mu$, supported in the closed unit disk. If there exists a real-valued function $f$ such that $$\mathbb{P}\left( |z-X_j|>t \right)\geq 1-f(t), $$ for all $z \in \mb{D}$ and $t$ small, then for any set $B\subset \mb{D}$, we have
        \begin{align}
             &\mathbb{P}\left(  \min_{2\leq j \leq n}|X_1-X_j| >t  \Big| X_1 \in B\right) \geq \left(1-f(t)\right)^n.
        \end{align}
    \end{Lemma}
    \begin{proof}[ \textbf{Proof of Lemma \ref{distance between roots}}]
    We use the independence of the random variables after conditioning on $X_1$ to write
            \begin{align*}
                &\mathbb{P}\left(  \min_{2\leq j \leq n}|X_1-X_j| >t \Big|X_1 \in B\right)\\[.5em]
                &= \frac{1}{\mathbb{P}( X_1 \in B)}\bigintsss_{\mathbb{D}} \mathbb{P}\left(  \min_{2\leq j \leq n}|X_1-X_j| >t,X_1 \in B \Big|X_1=z \right) d\mu(z)\\[.75em]
                &=\frac{1}{\mathbb{P}( X_1 \in B)}\bigintsss_{B} \mathbb{P}\left(  \min_{2\leq j \leq n}|z-X_j| >t\right) d\mu(z)\\[.75em]
                &=\frac{1}{\mathbb{P}( X_1 \in B)}\bigintsss_{B} \mathbb{P}\left( |z-X_j|>t\right)^{(n-1)} d\mu(z)\\[.75em]
                &\geq \frac{1}{\mathbb{P}( X_1 \in B)}\bigintsss_{B} \left(1-f(t)\right)^{(n-1)} d\mu(z)\\[.75em]
                &\geq \left(1-f(t)\right)^{(n-1)}.\qedhere
            \end{align*}  
    \end{proof}
    \begin{Lemma}\textbf{(Lower bound on first derivative)}\label{derivative L bound}
          Let $\{X_i\}_{i=1}^\infty$ be a sequence of i.i.d. random variables with law $\mu$, supported in the closed unit disk. Assume that for every $1\leq p \leq 3$, there exists some positive constant $C_p> 0$, such that $E\left[\left|\log|z-X_1|\right|^p\right] <C_p.$ Let $B_n \subset \mb{D}$ be such that for some $M \geq 0$ and $\forall z \in B_n $, we have $ E\left[|\log|z-X_1|\right] \geq -\frac{M}{\sqrt{n}}$. Then for $n$ large, there exists a constant $\hat{C}(M)>0$, depending on $M$ such that,
          \begin{align}\label{l36}
              \mathbb{P}\left(\big|{P}^{'}_n(X_1)\big| \geq e^{\sqrt{n}} \Big| X_1 \in B_n\right) \geq \hat{C}(M).
          \end{align}
    \end{Lemma}
    \begin{proof}[ \textbf{Proof of Lemma \ref{derivative L bound}}]
    We start by taking the logarithm to write
    \begin{align}\label{l21}
                \mathbb{P}\left(\big|{P}^{'}_n(X_1)\big| \geq e^{\sqrt{n}} \Big| X_1 \in B_n \right) &= \mathbb{P}\left(\prod_{j=2}^n |X_1-X_j|\geq e^{\sqrt{n}} \Big| X_1 \in B_n\right) \nonumber \\
                &= \mathbb{P}\left(\sum_{j=2}^n \log|X_1-X_j|\geq \sqrt{n}\Big| X_1 \in B_n\right)\nonumber \\
                &= \frac{\mathbb{P}\left(\sum_{j=2}^n \log|X_1-X_j|\geq \sqrt{n},{ X_1 \in B_n}\right)}{\mathbb{P}\left( X_1 \in B_n\right)} \nonumber \\
                &= \frac{1}{\mathbb{P}( X_1 \in B_n)}\bigintss_{B_n} \mathbb{P}\left(\sum_{j=2}^n \log|z-X_j|\geq \sqrt{n} \right)d\mu(z).
            \end{align}
    We estimate the probability inside the integral in (\ref{l21}) using Berry–Esseen theorem \eqref{BE} to arrive at
            \begin{align}\label{l22}
               &\frac{1}{\mathbb{P}( X_1 \in B_n)} \bigintss_{B_n} \mathbb{P}\left(\sum_{j=2}^n \left(\log|z-X_j|-\mathbb{E}[\log|z-X_j|]\right)\geq \sqrt{n}-(n-1)\mathbb{E}[\log|z-X_j|]\right)d\mu(z)\nonumber \\
                     &\geq  \frac{1}{\mathbb{P}( X_1 \in B_n)}\bigintss_{B_n} \mathbb{P}\left(\frac{1}{\sqrt{n}}\sum_{j=2}^n\left({\log|z-X_j|-\mathbb{E}[\log|z-X_j|]}\right)\geq (M+1) \right)d\mu(z)\nonumber \\
                      &\geq \frac{1}{\mathbb{P}( X_1 \in B_n)}\bigintss_{B_n} \left(\Phi\left(\frac{M+1}{\sigma(z)}\right)-\frac{C\rho(z)}{\sigma^3(z)\sqrt{n}}\right) d\mu(z),
            \end{align}  
            where $\sigma^2(z)=\mathbb{E}\left[\left(\log|z-X_j|\right)^2 \right]$, $\rho(z)=\mathbb{E}\left[\left|\log|z-X_j|\right|^3 \right]$ and $\Phi$ is the distribution function of standard normal. From the hypothesis, we have uniform upper and lower bounds on $\sigma^2(z)$ and $\rho(z)$ using which we can bound the integrand in (\ref{l22}) as 
            \begin{align}\label{l35}
                \Phi\left(\frac{(M+1)}{\sigma(z)}\right)-\frac{C\rho(z)}{\sigma^3(z)\sqrt{n}}
                \geq \left({C}_1(M)-\frac{C_2}{\sqrt{n}}\right).
            \end{align}
        Putting the bound (\ref{l35}) in the estimate (\ref{l22}) we get the required probability (\ref{l36}) for some absolute constant $\hat{C}$.
        \begin{align*}
            \mathbb{P}\left(|{P}^{'}_n(X_1)| \geq e^{\sqrt{n}}| X_1 \in B_n\right)&= \frac{1}{\mathbb{P}( X_1 \in B_n)}\bigintsss_{B_n} \left({C}_1(M)-\frac{C_2}{\sqrt{n}}\right)d\mu(z)
            \geq\Hat{C}(M). \qedhere
        \end{align*}
    \end{proof}
   
    \begin{Lemma}\textbf{(Bound on higher derivatives)}\label{higher derivative bounds}
         Let $\{X_i\}_{i=1}^\infty$ be a sequence of i.i.d. random variables with law $\mu$, supported on the closed unit disk. If there exists a constant $C>0$, such that $\mb{E}\left[ \frac{1}{|z-X_1|} \right] <C $ for all $z\in \mb{D}$, then for any $\mb{B} \subset \mb{D}$
             \begin{align}
                \mb{E} \left [ \frac{1}{k!}\left|\frac{P^{(k)}_n(X_1)}{{P}^{'}_n(X_1)}\right| \big| X_1 \in \mb{B} \right] \leq \binom{n-1}{k-1}C^{k-1}.
             \end{align}
    \end{Lemma} 
    \begin{proof}[\textbf{Proof of Lemma \ref{higher derivative bounds}}]
    We write ${P_n}(z)$ as $ {P_n}(z)=(z-X_1)Q_n(z)$, where $Q_n(z):=\prod_2^n(z-X_j),$ then differentiation yields,
    $${P}^{(k)}_n(z)=k{Q}^{(k-1)}_n(z)+(z-X_1)Q^{k}_n(z).$$ Putting $z=X_1$ in the above equation, we get $\frac{ {P}^{k}_n(X_1)}{{P}{'}_n(X_1)}=\frac{k{Q}^{(k-1)}_n(X_1)}{Q_n(X_1)}$. Since $X_1$ is not a root of $Q_n(z)$,  $\frac{{Q}^{(k-1)}_n(X_1)}{Q_n(X_1)}$ will have $(n-1)(n-2)...\big(n-(k-1)\big)$ many summands of the form $\left[\frac{1}{(X_1-X_2)...(X_1-X_{k})} \right]$. Here, we only care about the number of summands because after conditioning on $X_1$, all of them will have the same expected value.
            \begin{align*}
            \mb{E} \left [ \frac{1}{k!}\left|\frac{P^{(k)}_n(X_1)}{{P}^{'}_n(X_1)}\right|\big| X_1 \in \mb{B} \right] &\leq \bigintsss_{\mb{B}} \frac{1}{k!}\mb{E}\left(\left|\frac{P^{(k)}_n(X_1)}{{P}^{'}_n(X_1)}\right| \Big| X_1=z\right) \frac{d\mu(z)}{\mu(\mb{B})}\\
                &\leq \bigintsss_{ \mb{B}} \frac{k(n-1)(n-2)...(n-k+1)}{k!}\mb{E}\left(\left|\frac{1}{(z-X_2)...(z-X_{k})}\right|\right) \frac{d\mu(z)}{\mu(\mb{B})}\\
                &\leq  \binom{n-1}{k-1} \bigintsss_{\mb{B} } \left(\mb{E}\left|\frac{1}{(z-X_2)}\right| \right)^{k-1} \frac{d\mu(z)}{\mu(\mb{B})}\\
                &\leq \binom{n-1}{k-1} C^{k-1},
            \end{align*}
    where we got the last estimate using the hypothesis of the lemma.
    \end{proof}
We will need one last lemma from complex analysis which relates the number of components of a polynomial lemniscate with the number of critical points with critical value bigger or equal to $1$.
        \begin{Lemma}\label{Components and critical points}
        Let $Q_n(z)$, $\Lambda(Q_n)$ be as in \eqref{deterministic poly}, and $\{\beta_j\}_{j=1}^{n-1}$ be the set of critical points of $Q_n$. Then,
        \begin{align*}
           C(\Lambda)=1+ \left|\left\{\beta_j: |P(\beta_j)| \geq 1\right \}\right|.
        \end{align*}    
    \end{Lemma}
    \begin{proof}
        Let us assume that $C(\Lambda)=m$, i.e. there are $m$ many components of the lemniscate. Let $n_1,...,n_m$ be the number of zeroes in each of the components. We know that for a simple closed level curve of $f(z),$ say $\mathcal{C}$ if $f(z)$ is analytic up to the boundary of $\mathcal{C}$ and has $n$ zeroes inside $\mathcal{C}$, then $f{'}(z)$ has $(n-1)$ zeros inside it. The proof of this result can be found in  \cite{theoryoffunctionstitchmarsh}, Proposition $3.55$. Then the component containing $n_i$ many zeroes will have $(n_i-1)$ many critical points inside the component. Since all these critical points are inside the lemniscate, all the critical values are strictly less than 1. Therefore, the following algebraic manipulations yield the required equality.
        \begin{align*}
            \left|\left\{\beta_j: |P(\beta_j)| \geq 1\right \}\right| &= (n-1)- \left|\left\{\beta_j: |P(\beta_j)| < 1\right \}\right|\\
            &= (n-1)- \sum_{i=1}^m (n_i-1)=(m-1) \qedhere
        \end{align*}
        
    \end{proof}

\section{Proof of theorem \ref{component uniform on disk}}

    \begin{proof}[\textbf{Proof of theorem \ref{component uniform on disk}}]\textbf{(Lower bound)}  The proof of the lower bound in both the theorems is based on an estimate of the number of isolated components. We start by defining what we mean by an isolated component of a polynomial. Let $Q_n(z)$ be defined as in (\ref{deterministic poly}), then we say that a root $z_j$ forms an \emph{isolated component} if there exists a ball $\mathcal{B}$ containing $z_j$ such that, 
            \begin{align}\label{isolated}
                 \begin{cases}
                    & z_k \notin \mathcal{B}, \hspace{.57in} \forall k \neq j \\[10pt]
                    &|Q_n(z)| \geq 1,  \hspace{.25in}  \forall z \in\partial\mathcal{B}.
                 \end{cases}
            \end{align}
    The key observation here is that bounds on the derivatives at the root provide a sufficient condition for an isolated component. Suppose for the root $z_1$ there exists some $r>0$ such that the following holds, 
    \begin{align}\label{sufficient condition for component}
        \begin{cases}
            &|{Q}^{'}_n(z_1)\frac{r}{2}| \geq 1 \\[10pt]
            & \Big|\frac{Q^{(k)}_n(z_1)\frac{r^k}{k!}}{{Q}^{'}_n(z_1)\frac{r}{1!}}\Big| 
            < \frac{1}{2n^2} , \hspace{.25in}  \textit{for } k=2,...,n \\[10pt]
            &\underset{2\leq j  \leq n}{\min}|z_1-z_j|>r.
        \end{cases}
    \end{align}
    Then using Taylor series expansion of $Q_n(z)$ for $z\in \partial B(z_1,r)$ we get,
     \begin{align}\label{calc for sufficient cond} 
             |Q_n(z)|
            &\geq \left|Q^{'}_n(z_1){r}\right| - \sum_{k=2}^n \left|Q^{(k)}_n(z_1)\frac{r^k}{k!}\right|\nonumber\\
            & \geq\Big|Q^{'}_n(z_1)r\Big|\left(1- \sum_{k=2}^n\frac{|Q^{(k)}_n(z_1)\frac{r^k}{k!}|}{|Q^{'}_n(z_1)\frac{r}{1!}|}\right)\nonumber \\
            &\geq \Big|Q^{'}_n(z_1)r\Big| \left(1- \sum_{k=2}^n \frac{1}{2n^2}\right) \nonumber \\
            &\geq \left|Q^{'}_n(z_1)\frac{r}{2}\right| \geq 1.
        \end{align}
    This ensures that there is a connected component of the lemniscate inside the disk $B(z_1,r)$. 
    We now define for each $1\leq i\leq n,$ the event $L_i=\{X_i\hspace{0.05in}\mbox{forms an isolated component}\}$. Then it immediately follows that
            \begin{align*}
                \mb{E}\left[C(\Lambda_n)\right]&\geq \mb{E}\left[\sum_{i=1}^n \mathbbm{1}_{L_i}\right]\geq n\mb{E}\left[\mathbbm{1}_{L_i}\right]
                \geq n\mb{P}\left( L_1 \right).
            \end{align*} 
        Since the isolated roots are near the unit circle with high probability we only consider roots lying in $A_n:=\left \{z:1-\frac{1}{\sqrt{n}}<|z|<1 \right \}$.
            \begin{align}\label{a2}
                \mb{E}[C(\Lambda_n)]&\geq n\mb{P}\left( L_1|X_1 \in A_n \right)\mathbb{P}( X_1 \in A_n)
                \geq \sqrt{n}\mb{P}\left( L_1|X_1 \in A_n \right).
            \end{align}
    We now define the following events with $r_n=\frac{1}{n^6},$
        \begin{align}\label{events for lower bound}
            \begin{cases}
                &G_1:=\left \{ |{P}^{'}_n(X_1)| \geq e^{n^{1/2}} \right \}\\[1em]
                &G_k:=\left \{ \left|\frac{P^{(k)}_n(X_1)\frac{r_n^k}{k!}}{{P}^{'}_n(X_1)\frac{r_n}{1!}}\right| < \frac{1}{2n^2}  \Big|\right \}, \textit{  for } k =2,...,n.\\[1.5em]
                &G_{n+1}:=\left \{ \underset{{2\leq j \leq n}}{\min}|X_1-X_j| >\frac{1}{n^6} \right \}.
            \end{cases}
        \end{align}
    \noindent In the setting of (\ref{sufficient condition for component}), occurrence of the events \eqref{events for lower bound} implies that $X_1$ forms an \emph{isolated component}. Hence, 
        \begin{align}\label{prob bound}
           \mb{P}\left( L_1 |X_1 \in A_n \right)
            \geq \mb{P}\left(\cap_{j=1}^{n+1}G_j \big|X_1 \in A_n\right)
        \end{align}
    Now we will estimate the conditional probabilities of $G_1,{G_2},...,{G_{n+1}}$ one by one. From Lemma \ref{derivative L bound} we have
    \begin{align}\label{a5}
        \mb{P}\left(G_1|X_1 \in A_n\right)\geq C_1.
    \end{align}
    Using the Lemma \ref{higher derivative bounds} with $\mb{B}=A_n$ and the uniform bound of moment from Lemma \ref{moment bound for uniform on disk}, we get for $k=2,...,n$
    \begin{align}\label{a3}
        \mb{E}\left(\left|\frac{P^{(k)}_n(X_1)\frac{r_n^k}{k!}}{{P}^{'}_n(X_1)\frac{r_n}{1!}}\right| \Big|X_1 \in A_n\right) & \leq \frac{C^{k-1}}{n^{4(k-1)}}.
    \end{align}
    Now conditional Markov inequality with (\ref{a3}) yields,
     \begin{align}\label{a6}
        \mb{P}\left(G_k^c|X_1 \in A_n\right)\geq \mb{P}\left(\left|\frac{P^{(k)}_n(X_1)\frac{r_n^k}{k!}}{{P}^{'}_n(X_1)\frac{r_n}{1!}}\right|\geq \frac{1}{2n^2} \Big|X_1 \in A_n\right)  \leq \frac{1}{n^{2(k-1)}} 
    \end{align}
    Lastly, the Lemma \ref{distance between roots} with $t=\frac{1}{n^6}, f(x)=x^2$, and $C=1$ gives,
    \begin{align}\label{a4}
        \mb{P}\left(G_{n+1}|X_1 \in A_n\right) \geq \left(1-\frac{1}{n^{12}}\right)^{n-1} \geq 1-\frac{1}{n^{10}}
    \end{align}
    Combining the estimates  (\ref{a5}), (\ref{a6}), (\ref{a4}), we arrive at
        \begin{align}\label{a1}
         \mb{P}\big(\cap_{j=1}^{n+1}G_j |X_1 \in A_n\big)  &\geq \mb{P}(G_1|X_1 \in A_n)-\mb{P}\Big(  G_1 \cap \big(\cap_{k=2}^{n+1} {G_k}\big)^c |X_1 \in A_n\Big)\nonumber \\
             &\geq  C_1-\mb{P}\Big( \cup_{k=2}^n \big(G_1 \cap {G_k}^c|X_1 \in A_n\big)\Big)\nonumber \\
             & \geq C_1-\sum_{k=2}^n\mb{P}\Big( {G_k}^c |X_1 \in A_n\Big)\nonumber \\
             & \geq C_1-\frac{1}{n},
        \end{align}
    where we have used $\mb{P}(A\cap B)=\mb{P}(A)-\mb(A\cap B^c)$ in the first step and the union bound in the third step.
    Finally, putting (\ref{a1}) in (\ref{a2})  the required bound is obtained.
    \begin{align*}
        \mb{E}[C(\Lambda_n)] \geq \sqrt{n}\mb{P}\left( L_1|X_1 \in A_n \right) \geq \sqrt{n} \mb{P}\big( \cap_{j=1}^{n+1}G_j |X_1 \in A_n\big) \geq C_1\sqrt{n}.\\
    \end{align*}    
   \textbf{(Upper bound)} The proof of the upper bound uses Lemma \ref{Components and critical points} to relate the number of components to certain critical points.  We will take an indirect route to estimate the number of such critical points via the roots. We say a root $z_1$  of the polynomial $Q_n(z)$ is \emph{good}, if there exists $r>0$ such that,
          \begin{align}\label{Good root}
                 \begin{cases}
                    &B\left(z_1,r\right) \subset \Lambda_n,\\[.5em]
                    &\underset{2\leq j \leq n}{\min}|z_1-z_j| >3r, \\[1em]
                    &\exists \textit{ a unique critical point } \xi \in B\left(z_1,r\right).
                \end{cases}
             \end{align}   
        Resembling the proof of lower bound, we first give a sufficient condition for the ball of radius $ r_n:=\frac{1}{n^{3/4}} $ around $z_1$ to be inside the lemniscate. Assume the following holds, 
        \begin{align}\label{suff cond for ball outside the lemniscate}
        \begin{cases}
            &0<|{Q}^{'}_n(z_1)| < e^{-\sqrt{n}},\\[1em]
            & \left|\frac{Q^{(k)}_n(z_1)\frac{r_n^k}{k!}}{{Q}^{'}_n(z_1)\frac{r_n}{1!}}\right| < n^2 \binom{n-1}{k-1}\left(\frac{C}{n^{3/4}} \right)^{k-1}, \quad 2\leq k \leq n.
        \end{cases}
        \end{align}
    Then for $z \in \partial B\left(z_1, r_n\right)$  and  $n$ large enough, using (\ref{suff cond for ball outside the lemniscate}) we have, 
        \begin{align*}
             |Q_n(z)|& \leq  \left|Q^{'}_n(z_1)\frac{r_n}{1!}\right|+\left|Q^{''}_n(z_1)\frac{r_n^2}{2!}\right|+...+\left|Q^{(k)}_n(z_1)\frac{r_n^k}{k!}\right|+...+\left|Q^{(n)}_n(z_1)\frac{r_n^n}{n!}\right|\\[.75em]
            & \leq\left|Q^{'}_n(z_1)r_n\right|\left(1+\sum_{k=2}^n\frac{|Q^{(k)}_n(z_1)\frac{r_n^k}{k!}|}{|Q^{'}_n(z_1)\frac{r_n}{1!}|}\right)\\[.5em]
            &\leq \left|Q^{'}_n(z_1)r_n\right| \left(1+\sum_{k=2}^n  n^2\binom{n-1}{k-1}\left(\frac{C}{n^{3/4}} \right)^{k-1}\right)\\[.75em]
            &\leq n^2 e^{-\sqrt{n}} \left( 1+ \frac{C}{n^{3/4}} \right)^{n-1}\\[.75em]
            &\leq n^2 e^{-\sqrt{n}} e^{Cn^{1/4}}<1.
        \end{align*}
    The maximum principle then ensures that the disk $B(z_1,r)$ is inside the lemniscate.
    Let us now go back to the random setting and define the events $T_i:=\Big\{ X_i \mbox{ is a \emph{good} root with r}=\frac{1}{n^{3/4}} \Big\}$.
    The conditions in (\ref{Good root}) immediately imply that the number of \emph{good} roots is less than or equal to the number of critical points with critical value less than $1$, therefore,
            \begin{align*}
               \mb{E}[C(\Lambda_n)]&=n-\mb{E}\left[\{ \textit{Number of critical points with critical value less than } 1 \}\right]+1\\
               &\leq n-\mb{E}\left[ \sum_1^{n}\mathbbm{1}_{T_i}\right]+1
               \leq  n\left(1-\mb{P}(T_1)\right)+1.
            \end{align*}
        By concentration estimates, we expect that most of the \emph{good} roots are within the annulus $\mathbb{D}_n:=  \{ z: \frac{3}{n^{1/4}} <|z|\leq 1-\frac{1}{\sqrt{n}}\} $. So we estimate
        \begin{align}\label{b1}
             \mb{E}[C(\Lambda_n)] \leq n\left(1-\mb{P}\left(T_1|X_1 \in \mb{D}_n\right)\mb{P}(X_1 \in \mb{D}_n)\right)+1.
        \end{align}
    Now let us define the events $H_1,...,H_{n+1}$ with $ r_n:=\frac{1}{n^{3/4}} $.  
        \begin{align}\label{events H1}
            \begin{cases}
                &H_1:=\left \{ \left|{P}^{'}_n(X_1)\right| < e^{-\frac{\sqrt{n}}{2}}  \right\} \\[1.3em]
            &H_k:=\left \{ \left|\frac{P^{(k)}_n(X_1)\frac{r_n^k}{k!}}{{P}^{'}_n(X_1)\frac{r_n}{1!}}\right| < n^2 \binom{n-1}{k-1}\left(\frac{C}{n^{3/4}} \right)^{k-1} \right \}, \textit{  for } k =2,...,n.\\[1.75em]
            &H_{n+1}:=\left\{ \underset{2\leq j \leq n}{\min}|X_1-X_j| >3r_n \right \} \\[1.5em]
            &H_{n+2}:=\big\{\exists \textit{ a unique critical point } \xi \in B\left(X_1,r_n\right)\big            \}.
            \end{cases}
        \end{align}
    Notice that on the events \eqref{events H1}, we have a good root. Therefore
        \begin{align}\label{b2}
            \mb{P}\left(T_1|X_1 \in \mb{D}_n\right) \geq \mb{P}\left(\cap_{j=1}^{n+2}H_j|X_1 \in \mb{D}_n\right).
        \end{align}
    
    Next, we estimate the conditional probabilities of each of the events ${H_{1}},...,{H_{n+2}}$. To estimate the probability of the event ${H_{1}}$ we require the following lemma.
    \begin{Lemma}\textbf{( Upper bound on the first derivative)}\label{derivative U bound}
     Let $\{X_i\}_{i=1}^\infty$ be a sequence of i.i.d. uniform  random variables in the open unit disk. Let $
     \mathbb{D}_n:=\left\{ z: \frac{3}{n^{1/4}}<|z|\leq 1-\frac{1}{\sqrt{n}}\right\}$. Then there exists a constant $C>0$ such that,
        \begin{align}
              \mathbb{P}\left(|{P}^{'}_n(X_1)| \leq e^{-\frac{\sqrt{n}}{2}} \big| X_1 \in D_n \right) \geq 1-\frac{C}{\sqrt{n}}.
        \end{align}
    \end{Lemma}
    \begin{proof}[\textbf{Proof of Lemma \ref{derivative U bound}}]
   This proof adopts a methodology similar to Lemma \ref{derivative L bound}, but with a slight variation. Instead of using a uniform bound for the integrand, we actually perform the integration to achieve the desired inequality. We have
                \begin{align}\label{l31}
                    \mathbb{P}\left(\big|{P}^{'}_n(X_1)\Big| \geq e^{-\frac{\sqrt{n}}{2}} \big| X_1 \in \mathbb{D}_n \right) &= \mathbb{P}\left(\prod_{j=2}^n |X_1-X_j|\geq e^{-\frac{\sqrt{n}}{2}} \Big| X_1 \in \mathbb{D}_n\right)\nonumber \\
                    &= \mathbb{P}\left(\sum_{j=2}^n \log|X_1-X_j|\geq{-\frac{\sqrt{n}}{2}}  \Big| X_1 \in \mathbb{D}_n\right)\nonumber \\
                    &= \frac{1}{\mathbb{P}( X_1 \in \mathbb{D}_n)}\bigintsss_{ \mathbb{D}_n} \mathbb{P}\left(\sum_{j=2}^n \log|z-X_j|\geq {-\frac{\sqrt{n}}{2}} \right)d\mu(z).
                \end{align} 
    We use Bennett's inequality (\ref{Ben}) after subtracting the mean in (\ref{l31}), with the uniform upper and lower bounds of $\mb{E}\left[ \log|z-X_j|^2 \right]$ from Lemma \ref{log moment bound for uniform on cicle} to obtain, 
    \begin{align}\label{l32}
         & \frac{1}{\mathbb{P}( X_1 \in \mathbb{D}_n)}\bigintss_{ \mathbb{D}_n} \mathbb{P}\left(\sum_{j=2}^n \left(\log|z-X_j|-\mb{E} \left[\log|z-X_j|\right]\right)\geq \frac{(n-1)(1-|z|^2)}{2}{-\frac{\sqrt{n}}{2}} \right)d\mu(z)\nonumber\\[.35em]
         &\leq \frac{1}{\mathbb{P}( X_1 \in \mathbb{D}_n)}\bigintsss_{ \mathbb{D}_n} \exp{\left(-C_1nh\left(  \frac{(n-1)(1-|z|^2)-\sqrt{n}}{2C_2(n-1)}\right)\right)}d\mu(z)\nonumber \\[.35em]
         &\leq \frac{1}{\pi \mathbb{P}( X_1 \in \mathbb{D}_n)} \bigintsss_{0}^{2\pi}\bigintsss_{\frac{3}{n^{1/4}}}^{1-\frac{1}{\sqrt{n}}}\exp{\left(-C_1nh\left(  \frac{(n-1)(1-r^2)-\frac{\sqrt{n}}{2}}{2C_2(n-1)}\right)\right)}r dr d\theta \nonumber \\[.35em]
         &\leq \frac{2}{\mathbb{P}( X_1 \in \mathbb{D}_n)} \bigintsss_{\frac{3}{n^{1/4}}}^{1-\frac{1}{\sqrt{n}}}\exp{\left(-C_1 n h\left(  \frac{(n-1)(1-r^2)-\sqrt{n}}{2C_2(n-1)}\right)\right)} r dr
    \end{align}
    To estimate the integral we do a change of variables of $(1-r^2)=s$ in (\ref{l32}) and use the fact that $ C_3 u^2 \leq h(u) \leq C_4 u^2, $  for all $u \in [0,1]$, for some constants $C_3, C_4>0$. Then \eqref{l32} becomes 
    \begin{align*}
      &\frac{2}{\mathbb{P}( X_1 \in \mathbb{D}_n)} \bigintsss_{\frac{2}{\sqrt{n}}-\frac{1}{n}}^{1-\frac{9}{\sqrt{n}}} \exp{\left(-C_1 n h\left(  \frac{(n-1)s-\sqrt{n}}{2C_2(n-1)}\right)\right)} ds\nonumber\\[.35em]
          &\leq  \frac{2}{\mathbb{P}( X_1 \in \mathbb{D}_n)} \bigintsss_{0}^1 \exp{\left(-C_1 n \left({s}-\frac{1}{2\sqrt{n}}\right)^2\right)} ds\\[.35em]
        &\leq  \frac{2}{\mathbb{P}( X_1 \in \mathbb{D}_n)} \bigintsss_{0}^\infty \exp{\left(-x^2\right)} \frac{dx}{C_1\sqrt{n}} \\[.35em]
        &\leq  \frac{C}{\sqrt{n}}.
    \end{align*}
We finish the proof by taking the probability of the complementary event.
    \end{proof}
Using Lemma \ref{derivative U bound} above we deduce that,
        \begin{align}\label{b3}
            \mathbb{P}\left(H_1 \big|X_1 \in \mb{D}_n \right)=\mathbb{P}\left(\left|{P}^{'}_n(X_1)\right| < e^{-\frac{\sqrt{n}}{2}} \Big| X_1 \in \mathbb{D}_n \right) \geq 1-\frac{C_1}{\sqrt{n}}.
        \end{align}
       
    Now we estimate $\mathbb{P}\left(H_k \big|X_1 \in \mb{D}_n \right)$ for $2 \leq k \leq n $. By taking $\mb{B}=\mb{D}_n$ in Lemma \ref{higher derivative bounds} and the uniform bound from Lemma \ref{moment bound for uniform on disk},  we arrive at
        \begin{align}\label{b4}
            \mb{E}\left[\left|\frac{P^{(k)}_n(X_1)\frac{r_n^k}{k!}}{{P}^{'}_n(X_1)\frac{r_n}{1!}}\right| \Bigg| X_1 \in \mathbb{D}_n \right] 
            &\leq \binom{n-1}{k-1}\left(\frac{C}{n^{3/4}} \right)^{k-1}.
        \end{align}
       
     Now conditional Markov inequality along with (\ref{b4}) gives, 
        \begin{align}\label{b5}
            \mb{P}\left(H_k\big|X_1 \in \mb{D}_n \right)= \mb{P} \left(\left|\frac{P^{(k)}_n(X_1)\frac{r_n^k}{k!}}{{P}^{'}_n(X_1)\frac{r_n}{1!}}\right| \geq n^2 \binom{n-1}{k-1}\left(\frac{C}{n^{3/4}} \right)^{k-1} \Big| X_1 \in \mathbb{D}_n \right)  \leq \frac{1}{n^2}. 
        \end{align}
    Using Lemma \ref{distance between roots} with $t=\frac{1}{n^{3/4}}$ and $f(x)=x^2$ we obtain,
        \begin{align}\label{b6}
            \mb{P}\left(H_{n+1}\big|X_1 \in \mb{D}_n\right)&=\mb{P}\left( \underset{2\leq j \leq n}{\min}|X_1-X_j| >\frac{3}{n^{3/4}}  \Big| X_1 \in \mathbb{D}_n \right) \nonumber \\[.75em]
            &\geq \left(1-\frac{1}{n^{3/2}}\right)^{n-1} \geq 1-\frac{2}{\sqrt{n}}.
        \end{align}       
    Lastly, the probability bound for the event $H_{n+2}$ follows from the following lemma.
    \begin{Lemma}\textbf{( Distance between roots and critical points )}\label{pairing}
         Let $\{X_i\}_{i=1}^\infty$ be a sequence of i.i.d. uniform random variables in the open unit disk. We define the random polynomial $P_n$ as in \eqref{01}. Let $\mb{D}_n := \{z: \frac{3}{n^{1/4}}< |z| < 1-\frac{1}{n^{1/2}}\}$, and $r_n = \frac{1}{n^{3/4}}.$ Then
         \begin{align}
            \mb{P}\left(\big\{\exists \textit{ a unique critical point  } \xi \in B\left(X_1,r_n\right)\big\} \big| X_1 \in \mb{D}_n\right) \geq 1-\frac{C}{\sqrt{n}}.
        \end{align}
     \end{Lemma}
     
     \begin{proof}[\textbf{Proof of Lemma \ref{pairing}}]
       The proof can essentially be deduced from ideas in \cite{Kabluchko2019}. We first condition on the location of $X_1$ and rewrite the probability as 
       \begin{multline}\label{kab}
            \mb{P}\left(\big\{\exists \textit{ a unique critical point }  \xi \in B\left(X_1,r_n\right)\big\} \big| X_1 \in \mb{D}_n\right) \\
            = \int_{\mb{D}_n}\mb{P}\left(\big\{\exists \textit{ a unique critical point   } \xi \in B\left(X_1,r_n\right)\big\} \big| X_1 = u \right) d\mu(u).
        \end{multline}
       Fixing $u \in \mb{D}_n$ we define the event 
        \begin{align}\label{kab00}
            \mathcal{E}_n(u) := \left\{\sup_{z\in \partial B(u,r_n)} \left| \frac{1}{n} \frac{P_n^{'}(z)}{P_n(z)} - f(u) \right| < |f(u)| \right\},
        \end{align}
     where $r_n =\frac{1}{n^{3/4}}$ and $f(z) :=\mb{E}\left[\frac{1}{z-X_2}\right]= \bar{z}$ is the Cauchy transform of the uniform probability measure on $\mb{D}$. On the event $\mathcal{E}_n(u)$, by Rouche's theorem  (c.f. \cite{conway}, pp.125-126) the difference between the number of zeros and critical points of $ \frac{P_n^{'}(z)}{P_n(z)}$ on $B(u,r_n)$ is same as the difference between the number of zeros and poles of the constant function $z \mapsto f(u)$, which is zero. Now we define another event $\mathcal{F}_n(u):=\{|X_2-u|>3r_n,...,|X_n-u|>3r_n\}$ which guarantees that there is only one root of $P_n$ inside $B(u,r_n)$, hence only one critical point inside $B(u,r_n)$. Following the idea of proof of Lemma \eqref{distance between roots} one can show that $\mb{P}(\mathcal{F}_n(u)) \geq 1-\frac{C}{\sqrt{n}}$, therefore,
    \begin{align}\label{kab0}
        \mb{P}\left(\big\{\exists \textit{ a unique critical point  }\xi \in B\left(u,r_n\right)\big\}\right)
        \geq \mb{P}\left(\mathcal{E}_n(u) \cap \mathcal{F}_n(u)\right) \geq \mb{P}\left(\mathcal{E}_n(u) \right) -\frac{C}{\sqrt{n}}.
    \end{align}
     Next, writing $ \frac{P_n^{'}(z)}{P_n(z)}$ as sum of i.i.d random variables with mean $f(z)$ in \eqref{kab00} we get,
    \begin{align*}
        \mb{P}\left(\mathcal{E}_n(u)\right)=\left\{\sup_{z\in \partial B(u,r_n)} \left| \frac{1}{n(z-u)} + \frac{1}{n}\sum_2^n\frac{1}{z-X_j} - f(u) \right| < |f(u)| \right\}.
    \end{align*}
    Let $\Tilde{z}_n$ be a sequence of complex numbers in $ B(u,r_n)$ converging to $u$. Now adding and subtracting $\frac{1}{n}\sum_2^n\frac{1}{\Tilde{z}_n-X_j}$ and $f(\Tilde{z}_n)$ we bound the probability from below as
    \begin{multline}\label{kab1}
        \mb{P}\left(\mathcal{E}_n(u)\right) \geq \mb{P}\left( \sup_{z\in \partial B(u,r_n)}\left| \frac{1}{n(z-u)}\right|+ \sup_{z,\Tilde{z}_n\in B(u,r_n)}\left|\frac{1}{n}\sum_2^n\left(\frac{1}{z-X_j}-\frac{1}{\Tilde{z}_n-X_j}\right)\right| \right. \\
        \left. + \left| \frac{1}{n}\sum_2^n\frac{1}{\Tilde{z}_n-X_j} -f(\Tilde{z}_n) \right| + |\overline{\Tilde{z}_n-u}| < |f(u)| \right).
    \end{multline}
    Notice that, the maximum of the first and last term in \eqref{kab1} $\max\left\{\sup_{z\in \partial B(u,r_n)}\left| \frac{1}{n(z-u)}\right|,|\overline{\Tilde{z}_n-u}| \right\} \leq  \frac{1}{n^{1/4}}$, whereas $|f(u)| \geq \frac{3}{n^{1/4}}$. Therefore by triangle inequality, we get 
    \begin{align}\label{kab1.9}
        |f(u)|-\sup_{z\in \partial B(u,r_n)}\left| \frac{1}{n(z-u)}\right|-|\overline{\Tilde{z}_n-u}| \geq \frac{|f(u)|}{3}.
    \end{align}
    Plugging the estimate \eqref{kab1.9} in \eqref{kab1} we arrive at,
    \begin{multline}\label{kab2}
        \mb{P}\left(\mathcal{E}_n(u)\right) \geq \\
        \mb{P}\left(\sup_{z,\Tilde{z}_n\in B(u,r_n)}\left|\frac{1}{n}\sum_2^n\left(\frac{1}{z-X_j}-\frac{1}{\Tilde{z}_n-X_j}\right)\right| + {\left| \frac{1}{n}\sum_2^n\frac{1}{\Tilde{z}_n-X_j} -f(\Tilde{z}_n) \right|} < \frac{|f(u)|}{3} \right).
    \end{multline}
    Now taking complimentary events and using the fact that $\mb{P}(a+b>2) \leq \mb{P}(a>1) + \mb{P}(b>1)$ we obtain,
    \begin{multline}\label{kab5}
         \mb{P}\left(\mathcal{E}_n(u)\right) \geq 1-\mb{P}\underbrace{\left(\sup_{z,\Tilde{z}_n\in B(u,r_n)}\left|\frac{1}{n}\sum_2^n\left(\frac{1}{z-X_j}-\frac{1}{\Tilde{z}_n-X_j}\right)\right| \geq \frac{|f(u)|}{6}   \right)}_{(\mathrm{I})}\\
         -\mb{P}\underbrace{ \left(  {\left| \frac{1}{n}\sum_2^n\frac{1}{\Tilde{z}_n-X_j} -f(\Tilde{z}_n) \right|} \geq \frac{|f(u)|}{6} \right)}_{(\mathrm{II})}.
    \end{multline}
    To estimate ${(\mathrm{I})}$, we first simplify it using the following change of variables $z'=z-u, z_n^{''}=\Tilde{z}_n-u, X_j^{'}=X_j-u$ to get
    \begin{multline}\label{kab3}
         (\mathrm{I}) = \mb{P}\left(\sup_{z', z_n^{''}\in B(0,r_n)}\left|\frac{1}{n}\sum_2^n\left(\frac{(z'-z_n^{''})}{(z'-X_j^{'})(z_n^{''}-X_j^{'})}\right)\right|\geq \frac{|f(u)|}{6} \right) \\
         \leq \mb{P}\left(\sup_{z', z_n^{''}\in B(0,r_n)}\frac{2r_n}{n}\left|\sum_2^n\left(\frac{1}{(z'-X_j^{'})(z_n^{''}-X_j^{'})}\right)\right|\geq \frac{|f(u)|}{6} \right).
    \end{multline}
    Now using Markov inequality and Lemma 5.9 from \cite{Kabluchko2019} with $r_n=s_n=\frac{1}{n^{3/4}}$ and $a_n=\frac{2r_n}{n}$ in \eqref{kab3} we get, 
    \begin{align}\label{kab6}
        (\mathrm{I}) \leq \frac{6}{|f(u)|}\left[4na_n\left( -2\pi C\log(2s_n)+\mathcal{O}(1)\right)+4n\pi s_n^2C \right] \leq \frac{C}{|f(u)|\sqrt{n}}.
    \end{align}  
    We use the bound $(5.40)$ in Lemma 5.11 from \cite{Kabluchko2019} with $p=1.5, \epsilon =\frac{|f(u)|}{6}$ and uniform bounds on the moments from Lemma \eqref{moment bound for uniform on disk} to estimate $(\mathrm{II})$.
 \begin{align}\label{kab7}
     (\mathrm{II}) \leq \frac{C}{|f(u)|^{3/2}\sqrt{n}} \left( \mb{E}\left|\frac{1}{\Tilde{z}_n-X_1}\right|^{1.5}+|f(\Tilde{z}_n)|^{1.5}\right)\leq \frac{C}{|f(u)|^{3/2}\sqrt{n}}.
 \end{align}
 Now inserting \eqref{kab6}, and \eqref{kab7},  in \eqref{kab0} we obtain,
 \begin{align*}
     \mb{P}&\left(\big\{\exists \textit{ a unique critical point  } \xi \in B\left(X_1,r_n\right)\big\} \big| X_1 \in \mb{D}_n\right)\\
     &\geq \int_{\mb{D}_n} \left(1-\frac{C}{|f(u)|^{3/2}\sqrt{n}}-\frac{C}{|f(u)|\sqrt{n}}-\frac{C}{\sqrt{n}}\right)d\mu(u)\\
     &\geq 1-\frac{C}{\sqrt{n}}-C_1\int_{\frac{1}{n^{3/4}}}^{1-\frac{1}{\sqrt{n}}}\left(\frac{C}{r^{3/2}\sqrt{n}}-\frac{C}{r\sqrt{n}}\right)rdr\\
     &\geq 1-\frac{C}{\sqrt{n}}. \qedhere
 \end{align*}
\end{proof}
    Applying the union bound along with the estimates (\ref{b3}), (\ref{b5}), (\ref{b6}), and (\ref{kab}) leads to
        \begin{align}\label{b8}
            \mb{P}\left(T_1|X_1 \in \mb{D}_n\right) \geq \mb{P}\left( \cap_{k=1}^{n+2}H_k\big|X_1 \in \mb{D}_n\right) \geq 1-\sum_{k=1}^{n+2}  \mb{P}\left({H_k}^c \big|X_1 \in \mb{D}_n \right)  \geq 1-\frac{C}{\sqrt{n}}.
        \end{align}
    Feeding (\ref{b8}) into (\ref{b1}) the required upper bound is obtained.
       \begin{align*}
           \mb{E}[C(\Lambda_n)] &\leq n\left(1-\mb{P}\left(T_1|X_1 \in \mb{D}_n\right)\mb{P}(X_1 \in \mb{D}_n)\right)+1\\
            &\leq n\left(1- \left(1-\frac{C}{\sqrt{n}}\right)\left(1-\frac{2}{\sqrt{n}}\right) \right)+1 \\
            &\leq C_2{\sqrt{n}}.\qedhere
       \end{align*}
    \end{proof}   
\section{Proof of Theorem \ref{uniform on circle}} 
    In a recent paper \cite{KLM}, Krishnapur, Lundberg, and Ramachandran have shown that the polynomial lemniscate for roots chosen uniformly from the unit circle is a truly random quantity that converges in distribution to a sub-level set of a certain Gaussian function. Here, we show that the expected number of components for such lemniscates is asymptotically $\frac{n}{2}$. 

    \begin{proof}[\textbf{Proof of the theorem \ref{uniform on circle} (lower limit)}]
            The proof of the lower bound in this case follows the same strategy as in the previous theorem. The definition of an isolated component remains unchanged, and our focus lies on determining the number of such components. However, we cannot follow the proof verbatim because in this case, $\mb{E}\left[\frac{1}{|z-X_j|}\right]=\infty$. Therefore we condition on the following event to bypass this problem. Let us define the event $A:= \left \{ \underset{2\leq j \leq n}{\min}|X_1-X_j| >\frac{1}{n^3}\right \}$, then by Lemma \ref{distance between roots} with $t=\frac{1}{n^3}$ and $f(x)=2x$ the probability of the event $A$ is
            \begin{align}\label{c1}
                \mb{P}(A)& =\mathbb{P}\left(\underset{2\leq j \leq n}{\min}|X_1-X_j| >\frac{1}{n^3} \right)\geq 1-\frac{2}{n^2}.
            \end{align}
        For $1\leq i\leq n,$ let us define the events
        $S_i:=\{X_i$ \emph{forms an isolated component}$\}$. Then it immediately follows that
            \begin{align*}
                \mb{E}\left[C(\Lambda_n)\right]\geq \mb{E}\left[\sum_{i=1}^n \mathbbm{1}_{S_i}\right]\geq n\mb{E}\left[\mathbbm{1}_{S_i}\right]
                \geq n\mb{P}\left( S_1 \right)\geq n\mb{P}\left( S_1 \cap A \right)  \geq n\mb{P}\left( S_1|A \right)-\frac{2}{n}.
            \end{align*} 
        Next, we define events $F_1,...F_{n+1}$  as follows.
        \begin{align}\label{1}
           \begin{cases}
                &F_1:=\left \{ \big|{P}^{'}_n(X_1)\big| \geq e^{n^{1/2-\epsilon}} \right \}\\[1em]
            &F_k:=\left \{ \Big|\frac{P^{(k)}_n(X_1)\frac{r_n^k}{k!}}{{P}^{'}_n(X_1)\frac{r_n}{1!}}\Big| < \frac{1}{2n^2}  \right\} \textit{ , for } k =2,...,n.\\[1.5em]
            &F_{n+1}:=\left \{ \underset{2\leq j \leq n}{\min}|X_1-X_j| > \frac{1}{n^6}\right \}
           \end{cases}
        \end{align}
    \noindent From the calculations of  (\ref{sufficient condition for component}) and (\ref{calc for sufficient cond}) it follows that on the events (\ref{1}), we have an isolated component. Hence
    \begin{align}
           \mb{P}\left( S_1 | A \right) \geq \mb{P}\left( \cap_{j=1}^{n+1}F_j |A\right).
    \end{align} 
    As before we will calculate the conditional probability of the events $F_j$. Taking logarithms and using Berry-Esseen theorem (\ref{BE}) as in Lemma \ref{derivative L bound}, along with uniform bounds on the moments gives
            \begin{align*}
                \mathbb{P}\left(F_1 |A\right)=\mathbb{P}\left(|{P}^{'}_n(X_1)| \geq e^{n^{1/2-\epsilon}} \big| A\right)& \geq 
                 \mathbb{P}\left(|{P}^{'}_n(X_1)| \geq e^{n^{1/2-\epsilon}}\right ) - \frac{1}{n^2} \geq \frac{1}{2}-\frac{\Hat{C}}{n^{\epsilon}},
            \end{align*}
    where we used that $\mb{P}(A\cap B)=\mb{P}(A)-\mb(A\cap B^c)$. Notice that, on the event $A$, we have for $2\leq k\leq n,$
        \begin{align*}
            \left|\frac{P^{(k)}_n(X_1)\frac{r_n^k}{k!}}{{P}^{'}_n(X_1)\frac{r_n}{1!}}\right| 
            & \leq \frac{1}{n^{6(k-1)}k!}\sum_{i_1,...,i_{k-1}}\frac{1}{|X_1-X_{i_1}|...|X_1-X_{k-1}|} \\[.5em]
            & \leq   \frac{1}{n^{6(k-1)}k!}\sum_{i_1,...,i_{k-1}} n^{3(k-1)} \\[.5em]
            & \leq  \frac{k(n-1)(n-2)...(n-k+1)n^{3(k-1)}}{n^{6(k-1)}k!} \\
            &\leq \frac{1}{2n^2}. 
            \end{align*}
    Therefore, $\mathbb{P}(F_k \cap A)=\mathbb{P}( A)$ and as a result for $2 \leq k \leq n$, we have $\mathbb{P}(F_k|A)=1$. Since $  A \subset F_{n+1}$, we get the conditional probability $\mathbb{P}(F_{n+1}|A)=1$. Now using these bounds in (\ref{c1}) we obtain,
    \begin{align*}
         &\mb{E}\left[C(\Lambda_n)\right]\geq n\mb{P}\left( S_1|A \right)-\frac{2}{n}\geq\frac{n}{2}-{Cn^{1-\epsilon}} \\
        &\implies \underset{{n \to \infty}}{\liminf} \hspace{.1in}\frac{\mb{E}\left[C(\Lambda_n)\right]}{n} \geq \frac{1}{2}.
    \end{align*}

\textbf{(Upper limit)} The pairing of zeros and critical points does not occur in general if the law of the random variable does not have a density. Therefore when $\mu$ is the uniform probability measure on $\mb{S}^1$, we can not proceed by exploiting the pairing result. We prove the upper limit by showing the number of components having less than $n^\epsilon$ roots is approximately $\frac{n}{2}$. Let $C_k(\Lambda_n)$ denote the number of components containing exactly $k$ roots. Then it immediately follows that
\begin{align}
   &\sum_1^n C_k(\Lambda_n) =C(\Lambda_n),\label{e1}\\
   &\sum_1^n k C_k(\Lambda_n) =n.\label{e2}
\end{align}
For $i=1,...,n$, fix an $\epsilon>0$ small and  define the events $D_i:= \big\{$There are at least $n^{\epsilon/2}$ many roots inside the component containing the root $X_i \big\}.$ Now we claim that,
\begin{align}\label{e3}
C(\Lambda_n) \leq n-\sum_1^n \mathbbm{1}_{D_i} +\sum_{k\geq n^{\epsilon/2}} C_k(\Lambda_n).
\end{align}
 Substituting (\ref{e1}) and (\ref{e2}) in (\ref{e3}), we have to verify that
 \begin{align*}
     \sum_1^n \mathbbm{1}_{D_i} \leq \sum_{k< n^{\epsilon/2}}(k-1) C_k(\Lambda_n) +\sum_{k\geq n^{\epsilon/2}}k C_k(\Lambda_n)
 \end{align*}
 Since all the quantities are non-negative it is enough to show that
 \begin{align}\label{e6}
     \sum_1^n \mathbbm{1}_{D_i} \leq \sum_{k\geq n^{\epsilon/2}}k C_k(\Lambda_n).
 \end{align}
    Let $X_{i_{1}}$ be a root that has more than $ n^{\epsilon/2}$ roots in the component containing it. Assume that $X_{i_2},..., X_{i_m}$ are the other roots in this component say $C$. Then clearly,
    \begin{align}\label{e7}
        \sum_{k=1}^m\mathbbm{1}_{D_{i_k}} = m.
    \end{align}
    Now choose another root from $\{X_1,...,X_n\} \backslash  \{X_{i_2},..., X_{i_m}\}$ such that it has more than $n^{\epsilon/2}$ roots in the component containing it. Continuing this process and adding equations like (\ref{e7}) we get (\ref{e6}).
    Since the total number of roots is $n$, we can obtain a bound on the rightmost term of (\ref{e3}) in the following way.
    \begin{align}\label{e8}
        n &= \sum_{k< n^{\epsilon/2}}k C_k(\Lambda_n)+ \sum_{k\geq n^{\epsilon/2}}kC_k(\Lambda_n) 
        \geq  \sum_{k\geq n^{\epsilon/2}} kC_k(\Lambda_n) \nonumber \\[1em]
        \implies  n^{1-\epsilon/2} &\geq \sum_{k\geq n^{\epsilon/2}}C_k(\Lambda_n).
    \end{align}
    After putting the estimate (\ref{e6}) and taking expectation in both sides of (\ref{e3}) we arrive at,
    \begin{align}\label{e17}
        \mb{E}[C(\Lambda_n)] \leq n-n \mb{P}(D_1) + n^{1-\epsilon/2}.
    \end{align}
    To calculate the probability of the event $D_1$, let us first calculate the probability of having at least $ n^{\epsilon/2}$ roots in the ball $B(rX_1,\Tilde{r})$, where $r:=1-\frac{1}{n^{1-\epsilon}}$, $ \Tilde{r}:=\frac{2}{n^{1-\epsilon}}$. For $i=2,...,n$, define the events $\mathcal{T}_i :=\{X_j \in B(rX_1,\Tilde{r})\}$, then by the Paley-Zygmund inequality,
    \begin{align}
        \mb{P}\left(\sum_2^n\mathbbm{1}_{\mathcal{T}_i} \geq  n^{\epsilon/2} \right) \geq \left(1-\frac{ n^{\epsilon/2}}{\mb{E}\left[\sum_2^n\mathbbm{1}_{\mathcal{T}_i}\right]}\right)^2\frac{\mb{E}\left[\sum_2^n\mathbbm{1}_{\mathcal{T}_i}\right]^2}{\mb{E}\left[|\sum_2^n\mathbbm{1}_{\mathcal{T}_i}|^2\right]}.\label{e10}\\[-1em]\nonumber
    \end{align}
    Using the rotation invariance of the measure we get, 
    \begin{align*}
            \mb{P}\left(X_j \in B\left(rX_1,\Tilde{r}\right)\right)
        = \int_0^{2\pi} \mb{P}\left(X_j \in B(rX_1,\Tilde{r})\big| X_1=e^{i\phi}\right) \frac{d\phi}{2\pi}=\mb{P}\left(X_j \in B(re_1,\Tilde{r})\right), 
    \end{align*}
    where $e_1 = (1, 0).$ Let us assume that $B(re_1,\Tilde{r})$ intersects the unit circle at points $W_1, W_2$ and the angle $\angle OW_1W_2 = \theta$, where $O$ is the origin. Then it is easy to see that $ \mb{P}(X_j \in B(rX_1,\Tilde{r}))= \theta $ and $ \mb{P}(X_j, X_k \in B(rX_1,\Tilde{r}))= \theta^2 $, for all $j,k \neq 1$. Then
    \begin{align}               
    &\mb{E}\left[\sum_2^n\mathbbm{1}_{\mathcal{T}_i}\right]=(n-1)\theta, \label{e91} \\[.5em]
     &{\mb{E}\left[\left(\sum_2^n\mathbbm{1}_{\mathcal{T}_i}\right)^2\right]}=(n-1)\theta +(n-2)(n-1)\theta^2.  \label{e92}
    \end{align}
     Equation \eqref{e91} and  \eqref{e92} along with Bernoulli's inequality yields,
    \begin{align}
       &\frac{\mb{E}\left[\sum_2^n\mathbbm{1}_{\mathcal{T}_i}\right]^2}{\mb{E}\left[|\sum_2^n\mathbbm{1}_{\mathcal{T}_i}|^2\right]}= \frac{(n-1)^2\theta^2}{(n-1)\theta +(n-2)(n-1)\theta^2}\geq 1- \frac{1}{\theta(n-1)}\geq 1-\frac{C}{n^{\epsilon}},\label{e9}
    \end{align}
    where we got the last inequality using elementary geometry in the following way. For n large, $\sin \left( \frac{\theta}{4}\right) \geq \frac{C}{n^{1-\epsilon}}$, utilizing this, we bound $(n-1)\theta$ as
    $$(n-1)\theta \geq 4(n-1)\sin\left( \frac{\theta}{4}\right) \geq {C}{n^{\epsilon}}.$$
    Plugging the bound (\ref{e9}) in (\ref{e10}) we have,
    \begin{align}\label{e12}
        \mb{P}\left(\sum_2^n\mathbbm{1}_{\mathcal{T}_i} \geq  n^{\epsilon/2} \right) \geq \left(1-\frac{ n^{\epsilon/2}}{\mb{E}\left[\sum_2^n\mathbbm{1}_{\mathcal{T}_i}\right]}\right)^2\left(1- \frac{C}{n^{\epsilon}}\right)\geq \left(1- \frac{C}{n^{\epsilon/2}}\right) .
    \end{align}
    If  the ball $B(rX_1,\Tilde{r})$ is inside the lemniscate and there are at least $n^{\epsilon/2}$ roots inside the ball $B(rX_1,\Tilde{r})$, then the connected component containing $X_1$ must have atleast $n^{\epsilon/2}$ roots inside it. Now all we need is to estimate the probability that the ball $B(rX_1,\Tilde{r})$ is inside the lemniscate, which follows from the next lemma.                      
     \begin{Lemma}\label{Lemma ball in lemnicsate}
         Let $\{X_i\}_{i=1}^\infty$ be a sequence of i.i.d. random variables which are uniformly distributed on the unit circle. Fix $\epsilon\in (0,\frac{1}{4})$ and define $r:=1-\frac{1}{n^{1-\epsilon}}$, $ \Tilde{r}:=\frac{2}{n^{1-\epsilon}}$. Then there exists a constant $C>0$, such that,
         \begin{align}\label{ball in lemnicsate}
             \mb{P}\left( B(rX_1,\Tilde{r}) \subset \Lambda_n\right) \geq \frac{1}{2}-\frac{C}{n^{\epsilon}}
         \end{align}
    \end{Lemma}
    
    \begin{proof}[\textbf{Proof of Lemma \ref{Lemma ball in lemnicsate}}]
    Define $\Tilde{Q}_n(z):=\frac{Q_n(z)}{(z-z_1)} $ and assume that for some $r_1,\Tilde{r_1}$ the following is satisfied.
    \begin{align}\label{l01}
        \begin{cases}
            &4\Tilde{r_1}<1,\\[1em]
            &|\Tilde{Q}_n(r_1z_1)| \leq \exp\left({-n^{\frac{1}{2}-\epsilon}}\right),\\[1em]
            &\left|\frac{\Tilde{Q}^{(k)}_n(r_1z_1)\Tilde{r_1}^k}{\Tilde{Q}_n(r_1z_1)}\right| \leq n\sqrt{(n-1)...(n-k)}\left( \frac{4}{n^{1-\epsilon}}\right)^{k/2}, \quad k\geq 1.
        \end{cases}
    \end{align}
    Then for $z \in \partial  B(r_1z_1,\Tilde{r})$ and n large, we have,
    \begin{align*}
        |Q_n(z)| &= |z-z_1||\Tilde{Q}_n(z)|\\[.75em]
        &\leq 2\Tilde{r} \left( |\Tilde{Q}_n(r_1z_1)|+\left|\Tilde{Q}^{'}_n(r_1z_1)\frac{\Tilde{r_1}}{1!}\right|+... +\left|\Tilde{Q}^{(k)}_n(r_1z_1)\frac{\Tilde{r_1}^k}{k!}\right|+...+\left|\Tilde{Q}^{(n-1)}_n(r_1z_1)\frac{\Tilde{r_1}^{(n-1)}}{(n-1)!}\right| \right)\\[1em]
        &\leq |\Tilde{Q}_n(r_1z_1)| \left( 1+ \sum_{k=1}^{n-1}\left|\frac{\Tilde{Q}^k_n(r_1z_1)}{\Tilde{Q}_n(r_1z_1)}\frac{\Tilde{r}^k}{k!}\right|\right)\\[1em]
        &\leq  \exp\left({-n^{\frac{1}{2}-\epsilon}}\right)\left( 1+ \sum_{k=1}^{n-1}\frac{n\sqrt{(n-1)...(n-k)}}{k!}\left( \frac{4}{n^{1-\epsilon}}\right)^{k/2}\right),
    \end{align*}
    where we got the last line using \eqref{l01}. Now taking $n$ commmon in the paranthesis above and using the Cauchy–Schwarz inequality one has,
    \begin{align*}
         &\leq n \exp\left({-n^{\frac{1}{2}-\epsilon}}\right)\left( 1+ \sum_{k=1}^{n-1}\frac{{(n-1)...(n-k)}}{k!}\left( \frac{1}{n^{1/2+\epsilon/2}}\right)^{k}\right)^{1/2}\left( 1+ \sum_{k=1}^{n-1}\frac{1}{k!}\left( \frac{4}{n^{1/2-3/2\epsilon}}\right)^{k}\right)^{1/2}\\[1em]
         &\leq n \exp\left({-n^{\frac{1}{2}-\epsilon}}\right)\left( 1+ \sum_{k=1}^{n-1}\binom{n-1}{k}\left( \frac{1}{n^{1/2+\epsilon/2}}\right)^{k}\right)^{1/2}\left( 1+ \sum_{k=1}^{\infty}\frac{1}{k!}\left( \frac{4}{n^{1/2-3/2\epsilon}}\right)^{k}\right)^{1/2}\\[1em]
         &\leq n \exp\left({-n^{\frac{1}{2}-\epsilon}}\right)\left( 1+ \left( \frac{1}{n^{1/2+\epsilon/2}}\right)\right)^\frac{(n-1)}{2} \exp\left( 2n^{-1/2+3/2\epsilon}\right)\\[1em]
         &\leq n \exp\left({-n^{\frac{1}{2}-\epsilon}}\right)\exp\left({n^{\frac{1}{2}-\epsilon/2}}\right)\exp\left( 2n^{-1/2+3/2\epsilon}\right) <1.
    \end{align*}
    This ensures that the disk $B(z_1,r)$ is inside the lemniscate. Now with $r:=1-\frac{1}{n^{1-\epsilon}}$, $ \Tilde{r}:=\frac{2}{n^{1-\epsilon}}$ and defining $\Tilde{P}_n$ similarly to $\Tilde{Q}_n$, we define the following events,
    \begin{align}
    \begin{cases}
        \mathcal{G}_1 &:= |\Tilde{P}_n(rX_1)| \leq \exp\left({-n^{\frac{1}{2}-\epsilon}}\right)\\[1em]
        \mathcal{G}_k &:= \left|\frac{\Tilde{P}^{(k)}_n(rX_1)\Tilde{r}^k}{\Tilde{P}_n(rX_1)}\right| \leq n\sqrt{(n-1)...(n-k)}\left( \frac{4}{n^{1-\epsilon}}\right)^{k/2} , \hspace{.25in}  \textit{for } k=2,...,n.
    \end{cases}
    \end{align}
    By the conditions in (\ref{l01}) it immediately follows that, 
    \begin{align}
         \mb{P}\left( B(rX_1,\Tilde{r}) \subset \Lambda_n\right) \geq \mb{P}\left( \cap_1^n \mathcal{G}_k\right).
    \end{align}
    Let us calculate the probabilities of the events $\mathcal{G}_1,...,\mathcal{G}_n$ individually. To estimate $\mb{P}\left(  \mathcal{G}_1 \right)$, we take logarithm, use the fact that the mean of this random variable is $0$, and apply the Berry-Esseen Theorem (\ref{BE}) as done in Lemma \ref{derivative L bound}. Then it follows that for some constant $C_1$, 
    \begin{align}
        \mb{P}\left(  \mathcal{G}_1 \right) \geq \frac{1}{2}-\frac{C_1}{n^{\epsilon}}.
    \end{align}
    For the events $\mathcal{G}_k,$ we use Chebyshev's inequality to obtain,
    \begin{align}\label{l02}
        &\mb{P} \left(\left|\frac{\Tilde{P}^{(k)}_n(rX_1)\Tilde{r}^k}{\Tilde{P}_n(rX_1)}\right|\geq n\sqrt{(n-1)...(n-k)}\left( \frac{4}{n^{1-\epsilon}}\right)^{k/2} \right) \nonumber \\
        & \hspace{1.75in}\leq \frac{1}{n^2{(n-1)...(n-k)}}\left( \frac{n^{1-\epsilon}}{4}\right)^{k}\Tilde{r}^{2k}\mb{E}\left[\left|\frac{\Tilde{P}^{(k)}_n(rX_1)}{\Tilde{P}_n(rX_1)}\right|^2\right].
    \end{align}
    We estimate $\mb{E}\left[\left|\frac{\Tilde{P}^{(k)}_n(rX_1)}{\Tilde{P}_n(rX_1)}\right|^2\right]$ using the following facts
    \begin{align}
        &\mb{E}\left[ \frac{1}{z-X_1}\right]=0, \hspace{.11in} \forall z \in \mb{D}, \label{mean 1}\\
        &\mb{E}\left[\frac{1}{|r-X_1|^2}\right]= \frac{1}{1-r^2}.\label{mean 2}
    \end{align}
    The identity \eqref{mean 1} follows from the Cauchy integral formula, and \eqref{mean 2} follows using standard integration  techniques.
    \begin{align*}
        &\mb{E}\left[\left|\frac{\Tilde{P}^{(k)}_n(rX_1)}{\Tilde{P}_n(rX_1)}\right|^2\right]
       =\mb{E}\left[\left|\sum_{2\leq i_1<i_2<...<i_k \leq n}\frac{1}{(rX_1-X_{i_1})...(rX_1-X_{i_k})}\right|^2\right]\nonumber\\
        =&\mb{E}\left[\sum_{2\leq i_1<...<i_k \leq n}\frac{1}{(rX_1-X_{i_1})...(rX_1-X_{i_k})}\sum_{2\leq j_1<...<j_k \leq n}\frac{1}{\overline{(rX_1-X_{j_1})}...\overline{(rX_1-X_{j_k})}}\right]\nonumber\\
        =&\frac{1}{2\pi}\bigintss_0^{2\pi}\mb{E}\left[\sum_{2\leq i_1<...<i_k \leq n}\frac{1}{(re^{i\theta}-X_{i_1})...(re^{i\theta}-X_{i_k})}\sum_{2\leq j_1<...<j_k \leq n}\frac{1}{\overline{(re^{i\theta}-X_{j_1})}...\overline{(re^{i\theta}-X_{j_k})}}\right] d\theta\nonumber\\
        =&\mb{E}\left[\sum_{2\leq i_1<...<i_k \leq n}\frac{1}{(r-X_{i_1})...(r-X_{i_k})}\sum_{2\leq j_1<...<j_k \leq n}\frac{1}{\overline{(r-X_{j_1})}...\overline{(r-X_{j_k})}}\right] \nonumber
    \end{align*}
    Notice that by the independence of the random variables, and identity \eqref{mean 1}, the cross terms will vanish. We estimate the remaining terms using \eqref{mean 2} in the following way.
    \begin{align}\label{l04}
        \mb{E}\left[\left|\frac{\Tilde{P}^{(k)}_n(rX_1)}{\Tilde{P}_n(rX_1)}\right|^2\right]=&\mb{E}\left[\sum_{2\leq i_1<...<i_k \leq n}\frac{1}{|r-X_{i_1}|^2...|r-X_{i_k}|^2}\right] 
        =(n-1)...(n-k)\mb{E}\left[\frac{1}{|r-X_1|^2}\right]^k \nonumber \\[.5em]
        \leq&(n-1)...(n-k) (1-r^2)^{-k}
        \leq 2^k(n-1)...(n-k) n^{k(1-\epsilon)}.
    \end{align}
   Now plugging the bound (\ref{l04}) in (\ref{l02}) and taking the complementary events we get,
   \begin{align}\label{l06}
       &\mb{P} \left(\left|\frac{\Tilde{P}^{(k)}_n(rX_1)\Tilde{r}^k}{\Tilde{P}_n(rX_1)}\right|\leq n\sqrt{(n-1)...(n-k)}\left( \frac{4}{n^{1-\epsilon}}\right)^{k/2} \right) \geq 1- \frac{1}{n^2}.
   \end{align}
    Making use of (\ref{l06}) and (\ref{l02}) in (\ref{l01}) we arrive at the required probability.
        \begin{align*}
            \mb{P}\left( B(rX_1,\Tilde{r}) \subset \Lambda_n\right) &\geq \mb{P}\left( \mathcal{G}_1\right) - \mb{P}\left( \mathcal{G}_1\cap \left(\cap_2^n \mathcal{G}_k\right)^c\right)\\
            &\geq \frac{1}{2}-\frac{C_1}{n^{\epsilon}}-\sum_2^n \frac{1}{n^2}\\
            &\geq \frac{1}{2}-\frac{C}{n^{\epsilon}}.\qedhere
        \end{align*} 
    \end{proof}
    Then using the bound \eqref{ball in lemnicsate} in Lemma \ref{Lemma ball in lemnicsate} and (\ref{e12}) we get the required probability.
    \begin{align}\label{e13}
        \mb{P}(D_1) &\geq \mb{P}\left( \left\{\sum_2^n\mathbbm{1}_{\mathcal{T}_i} \geq  n^{\epsilon/2} \right\} \bigcap \Big\{  B(rX_1,\Tilde{r}) \subset \Lambda_n \Big\}\right) \nonumber \\[1em]
        &\geq \frac{1}{2}-\frac{C}{n^{\epsilon}} - \frac{2C}{n^{\epsilon/2}} \geq \frac{1}{2}-\frac{C}{n^{\epsilon/2}}.
    \end{align}
    Now setting (\ref{e13}) in (\ref{e17}) and taking the limsup we get the asymptotic upper bound, i.e,
    \begin{align*}
         \underset{{n \to \infty}}{\limsup} \frac{\mb{E}[C(\Lambda_n)] }{n} \leq \underset{{n \to \infty}}{\limsup} \hspace{.07in}\frac{1}{n}\left[n-n\left(\frac{1}{2}-\frac{2C}{n^{\epsilon/2}}\right)  + n^{1-\epsilon/2} \right]
         \leq \frac{1}{2}.
    \end{align*}
\end{proof}
\subsection{Acknowledgement}The author expresses gratitude to his thesis advisor Dr. Koushik Ramachandran for suggesting the problem and for feedback on the article. The author deeply appreciates the support, encouragement, and numerous simulating conversations he received from his advisor throughout this project.

\bibliographystyle{siam}
\bibliography{references}

\end{document}